\documentclass[11pt,reqno]{amsart}

\usepackage[T1]{fontenc}
\usepackage[utf8]{inputenc}
\usepackage{mathpazo}
\usepackage{microtype}
\usepackage[a4paper,margin=31mm]{geometry}
\usepackage{amsmath,amssymb,amsthm,mathtools}
\usepackage{booktabs}
\usepackage{tabularx}
\usepackage{array}
\usepackage{ragged2e}
\usepackage{float}
\usepackage{graphicx}
\usepackage{esint}

\newcolumntype{Y}{>{\RaggedRight\arraybackslash}X}

\usepackage{enumitem}
\usepackage{xcolor}
\usepackage[colorlinks=true,linkcolor=blue!45!black,
  citecolor=blue!45!black,urlcolor=blue!45!black]{hyperref}
\usepackage[nameinlink,noabbrev]{cleveref}

\definecolor{darkred}{rgb}{0.7,0.1,0.1}

\allowdisplaybreaks[2]
\numberwithin{equation}{section}
\setlist{itemsep=2pt,topsep=4pt}

\newtheorem{theorem}{Theorem}[section]
\newtheorem{proposition}[theorem]{Proposition}
\newtheorem{lemma}[theorem]{Lemma}
\newtheorem{corollary}[theorem]{Corollary}
\newtheorem{definition}[theorem]{Definition}
\theoremstyle{remark}
\newtheorem{remark}[theorem]{Remark}

\newcommand{\R}{\mathbb R}
\newcommand{\Om}{\Omega}
\newcommand{\dd}{\,\mathrm d}
\newcommand{\cA}{\mathcal A}
\newcommand{\cR}{\mathcal R}
\newcommand{\esssup}{\operatorname*{ess\,sup}}
\newcommand{\essinf}{\operatorname*{ess\,inf}}
\newcommand{\med}{\operatorname{med}}
\newcommand{\supp}{\operatorname{supp}}
\newcommand{\Wulff}{\mathcal W}

\title[Anisotropic gradient rearrangement estimates]
{Gradient rearrangement estimates for singular and degenerate
anisotropic $p$-Laplace equations}

\author{Lu Chen}
\address{Key Laboratory of Algebraic Lie Theory and Analysis of Ministry of Education,
School of Mathematics and Statistics,
Beijing Institute of Technology,
Beijing 100081, P.~R.~China}
\email{chenlu5818804@163.com}

\author{Daniel Hauer}
\address{Institute of Mathematics,
Brandenburg University of Technology Cottbus--Senftenberg,
Platz der Deutschen Einheit 1,
03046 Cottbus, Germany}
\email{daniel.hauer@b-tu.de}

\address{School of Mathematics and Statistics,
The University of Sydney,
NSW 2006, Australia}
\email{daniel.hauer@sydney.edu.au}

\author{Yabo Yang}
\address{Key Laboratory of Algebraic Lie Theory and Analysis of Ministry of Education,
School of Mathematics and Statistics,
Beijing Institute of Technology,
Beijing 100081, P.~R.~China}
\email{Yabo\_Yang0927@outlook.com}

\date{26 Aug 2026}

\subjclass[2020]{35J25, 35B65, 35J92, 46E30}
\keywords{anisotropic $p$-Laplacian, singular elliptic equation, SOLA, gradient rearrangement, Lorentz space,
Dirichlet problem, Neumann problem}

\begin{document}

\begin{abstract}

In this paper, we prove global pointwise estimates on the decreasing rearrangement $u^{\ast}$ and $[H(\nabla u)]^{\ast}$ of SOLAs $u$ of an elliptic Poisson equation driven by the anisotropic $p$-Laplace operator $\Delta_p^H$ for given $L^1$ data and equipped with either homogeneous Dirichlet or Neumann boundary conditions. Here, we assume that the dimension $n\ge 2$, $1<p<n$ and $H$ is norm on $\R^n$. In the degenerate range $2 \leq p <n$,
this gives an anisotropic extension and a full-interval completion of the
rearrangement theory of Cianchi-Maz'ya [JEMS, 2014]. In particular, we obtain a new rearrangement estimate for gradients by decomposing the data for singular range $1<p<2$, where the nonlinear interpolation method aren't available. Moreover, the endpoint weak-type gradient estimate is obtained for all $1<p<n$. In the range $p_*(n)<p<n$ with $p_*(n):=\frac{3n-2}{2n-1}$, we prove a stability estimate for differences of gradients. These results extends the known ones in the literature to the singular range $p_*(n)<p<2$.

\end{abstract}

\maketitle

\section{Introduction}\label{sec:introduction}

Let $\Omega$ be a bounded domain in $\R^n$, $n\ge 2 $, and $H$ be a norm on $\R^n$. Then, our aim is to establish point-wise estimates on the decreasing rearrangement $[H(\nabla u)]^*$ of the gradient $\nabla u$ of
(approximable) solutions $u$ to the anisotropic Poisson equation
\begin{equation}
    \label{eq:Poisson}
 -\operatorname{div}\!\left(
 H(\nabla u)^{p-1}\nabla H(\nabla u)
 \right)+\beta (u) = f
 \qquad\text{in }\Omega
\end{equation}
satisfying either homogeneous Neumann or Dirichlet boundary conditions for given datum $f\in L^1(\Omega)$ and a given monotone (increasing) continuous function $\beta:\,\R\to \R$ satisfying $\beta(0)=0$ and certain growth conditions. The decreasing rearrangement $[H(\nabla u)]^*(r)$ for $r>0$ records the size of the anisotropic gradient $H(\nabla u)$ 
independently of its location. Thus, pointwise estimate for this
one-dimensional function is considerably more informative than a single
norm inequality: it simultaneously yields global regularity bounds in Lebesgue, Lorentz,
Marcinkiewicz, or Orlicz spaces. Obtaining such an estimate on the whole
interval $(0,|\Omega|)$, and up to the boundary of $\Omega$, is the main
purpose of this paper.

Rearrangement methods have a long history in elliptic boundary value problems.
Talenti's comparison principles \cite{Talenti1976,Talenti1979} compare the symmetric rearrangement of a
solution with the radial solution of a symmetrized problem and lead to sharp estimates
for the solution itself. In the anisotropic setting, Euclidean balls are replaced by
Wulff shapes, and convex symmetrization was developed in \cite{AlvinoFeroneLionsTrombetti1997}; related geometric
properties of anisotropic operators can be found in \cite{BelloniFeroneKawohl2003,CianchiSalani2009}. Fully anisotropic problems
with minimally integrable data were considered in \cite{AlbericoEtAl2019}. For Neumann problems,
ordinary isoperimetry has to be replaced by relative isoperimetry, and a median is the
natural normalization; see \cite{ACMM,Mazya1969}.

Pointwise rearrangement estimates for gradients in problems with merely integrable data were already obtained by Alvino-Ferone-Trombetti \cite{AlvinoFeroneTrombetti2000}. For homogeneous Dirichlet problems involving a broad class of nonlinear elliptic operators, they worked in the framework of entropy solutions and combined truncation estimates, level-set and coarea arguments, the isoperimetric inequality to derive an explicit pointwise bound for \(|\nabla u|^*\) in terms of \(f^*\). Building on this level-set and rearrangement strategy, Cianchi-Maz'ya \cite{CianchiMazya2014} subsequently obtained the sharper and more transparent estimate for \(p\)-Laplace type equations in the degenerate range \(2\leq p<n\). In particular, their estimate is naturally formulated in terms of \((|\nabla u|^*)^{p-1}\), has a linear Hardy-type dependence on \(f^*\), and applies to both homogeneous Dirichlet and Neumann problems. The present paper continues this line of research by extending gradient rearrangement estimates to anisotropic \(p\)-Laplace equations, with particular extend the range of $p$ to the singular range \(1<p<2\).

The gradient is substantially more delicate than the solution. A comparison principle
for $u$ does not, in general, differentiate into a pointwise comparison for $\nabla u$.
Gradient estimates must instead be extracted from truncations, level-set identities,
coarea formulas, isoperimetric inequalities, and endpoint bounds; see, for instance, \cite{AlvinoFeroneLionsTrombetti1997}. The low integrability
of the datum is a genuine source of difficulty. In general,
\[
 L^1(\Omega)\not\subset W^{-1,p'}(\Omega),
\]
and a solution need not belong to the natural energy space. Thus the
equation cannot always be understood through the standard weak formulation.
We work instead with approximable solutions (define as SOLA), obtained as limits of solutions
with bounded data. This framework, developed in
\cite{BenilanEtAl1995,BoccardoGallouet1989,DalMasoEtAl1999,DallAglio1996},
retains the almost-everywhere convergence of gradients needed in the
level-set argument.

A benchmark for the present work is the celebrated paper of Cianchi and Maz'ya \cite{CianchiMazya2014}.
For the isotropic equation
\[
 -\operatorname{div}\!\big(a(|\nabla u|)\nabla u\big)=f
\]
with $p$-growth in the degenerate range $2\le p<n$, they obtained a universal
rearrangement estimate under homogeneous Dirichlet or Neumann boundary conditions
and under borderline assumptions on the geometry of the domain. With
$m=|\Omega|$, $n'=n/(n-1)$, and
$f^{**}(r)=r^{-1}\int_0^r f^*(\rho)\,d\rho$, their main estimate is stated as
\begin{equation}\label{eq:cianchi-mazya-theorem}
 |\nabla u|^*(s)^{p-1}
 \le C\int_s^m f^{**}(r)r^{-\frac{1}{n'}}\,dr,
 \qquad 0<s<m.
\end{equation}
Its decisive feature is that it is independent of any prescribed function space. It
reduces gradient estimates in rearrangement-invariant norms to one-dimensional Hardy inequalities and thereby yields applications in Lebesgue, Lorentz, Lorentz--Zygmund, and Orlicz spaces. In particular, \cite{CianchiMazya2014} shows that the nonlinear
quantity $|\nabla u|^{p-1}$ obeys the same integrability pattern as the gradient in the linear theory.

Our results complement and sharpen this theory in three distinct directions. First, we
replace the Euclidean constitutive law by the genuinely anisotropic vector field
\[
 A_H(\xi)=H(\xi)^{p-1}\nabla H(\xi),
\]
and treat both the Dirichlet and Neumann problems. This requires
anisotropic perimeter, relative anisotropic isoperimetry, and quantitative convexity of
$H^2$; it is not a formal change of notation. Second, and more substantially, we cover
the entire singular interval $1<p<2$, where the difference estimates and nonlinear
interpolation used in the degenerate regime are no longer available. Third, our result
identifies the form of the Hardy tail that remains valid on the whole interval
$(0,m)$.

The last point deserves an explicit explanation. Put
\[
 A_f(s)=\int_0^s f^*(r)\,dr,
 \qquad
 B_f(s)=\int_s^m f^*(r)r^{-\frac{1}{n'}}\,dr.
\]
A direct application of Fubini theorem gives
\[
 \int_s^m f^{**}(r)r^{-\frac{1}{n'}}\,dr
 =n'\Big(s^{-\frac{1}{n'}}A_f(s)+B_f(s)-m^{-\frac{1}{n'}}\|f\|_{L^1(\Omega)}\Big).
\]
By contrast, if $f^*$ is extended by zero to $(m,\infty)$ and the corresponding
maximal rearrangement is denoted by $\widetilde f^{**}$, then
\begin{align}\label{eq:intro-classical}
 \int_s^\infty \widetilde f^{**}(r)r^{-\frac{1}{n'}}\,dr
 =n'\big(s^{-\frac{1}{n'}}A_f(s)+B_f(s)\big).
\end{align}
The two tails are uniformly comparable away from the terminal endpoint, for example
on $0<s\le m/2$, but they are not uniformly comparable as $s\uparrow m$. In fact, an
explicit solution with constant datum on a Wulff ball shows that the unshifted
finite-tail expression in \eqref{eq:intro-classical} decays too rapidly to control the gradient all the way to $s=m$. Thus, in the degenerate isotropic case, our theorem has the same function-space consequences as \cite{CianchiMazya2014}, while its two-term form supplies the full-interval formulation. This distinction is invisible in the usual
small-$s$ estimates but essential for a genuinely global pointwise rearrangement inequality. Moreover, the two terms in \eqref{eq:intro-classical}
reflect two endpoint phenomena: the weak-type estimate associated with
$L^1$ data and the global Lipschitz estimate associated with data in the Lorentz space $L^{n,1}$. The relevant global Lipschitz theory
was developed in \cite{CianchiMazya2011,CianchiMazya2014Systems}. In the
degenerate regime, the full rearrangement inequality can be recovered from
these endpoints through nonlinear interpolation and explicit formulas for
the relevant $K$-functionals, readers may refer to
\cite{BennettSharpley,Holmstedt1970,Tartar1972}.

Recent results have supplied important components of the corresponding
anisotropic theory. Interior gradient estimates and global second-order
estimates were obtained in
\cite{AntoniniCiraoloFarina2023,AntoniniEtAl2025}. Most importantly for
the present problem, Antonini and Cianchi \cite{AntoniniCianchi2025}
established the anisotropic gradient endpoint estimate of $L^{n,1}$ and
$L^\infty$ for every $p>1$, both on convex domains and under the borderline
boundary condition described above. The results of Dong-Zhu \cite{DongZhu2024} are formulated in terms of nonlinear potentials of the
measure datum and yield pointwise gradient bounds and continuity estimates.
To the best of our knowledge, however, a global estimate for the complete
decreasing rearrangement of the gradient, simultaneously covering
homogeneous Dirichlet and Neumann problems, has not previously
been available in the anisotropic setting.

There are two nonformal obstructions to such an extension. The first is
geometric: the proof must use anisotropic perimeter, relative anisotropic
isoperimetry, and the quantitative convexity of $H^2$. The second is
specific to the singular range $1<p<2$. If
\[
 \cA_{H}(\xi)=H(\xi)^{p-1}\nabla H(\xi),
\]
then for $p\geq2$ one has
\[
 \big(\cA_{H}(\xi)-\cA_{H}(\eta)\big)\cdot(\xi-\eta)
 \gtrsim H(\xi-\eta)^p,
\]
whereas for $1<p<2$ the natural coercivity inequality becomes
\begin{equation}\label{eq:intro-weighted}
 \big(\cA_{H}(\xi)-\cA_{H}(\eta)\big)\cdot(\xi-\eta)
 \gtrsim
 [H(\xi)+H(\eta)]^{p-2}H(\xi-\eta)^2.
\end{equation}
Consequently, the difference of two gradients can no longer be controlled
independently of the two background solutions.

There is also a structural obstruction to continuing the degenerate
interpolation argument. The gradient operator 
$T:f\mapsto\nabla u_f$ of the solution mapping $f\mapsto u_f$ has degree of homogeneity $\frac{1}{p-1}$. If $p<2$,
this degree is larger than one, and an estimate of the form
\[
 \|T(f)-T(g)\|_X
 \leq
 C\|f-g\|_{L^1(\Omega)}^{\frac{1}{p-1}}
\]
would contradict homogeneity along every nonzero ray. Thus the stability
estimate underlying the interpolation argument in
\cite{CianchiMazya2014} cannot hold in the singular regime.

Our results may be summarized in the following principles.

\begin{itemize}

\item A global rearrangement estimate for $H(\nabla u)$ holds throughout
the full interval $1<p<n$, for both homogeneous Dirichlet and Neumann boundary conditions.

\item The endpoint weak-type estimate
\[
 T:\; L^1(\Omega)
 \longrightarrow
 L^{\frac{n(p-1)}{n-1},\infty}(\Omega)
\]
remains valid for every $1<p<n$.

\item In a subrange of $1<p<2$, differences of gradients satisfy a
background-dependent stability estimate with the precise homogeneity
dictated by the equation.
\end{itemize}

Our proof separates the estimate for one solution from estimates for differences of
solutions. A direct level-set argument first gives the weak endpoint bound for a single
solution for every $1<p<n$. It uses only the identity
\[
 \cA_H(\nabla u)\cdot\nabla u=H(\nabla u)^p,
\]
and subadditivity of rearrangement inequalities
\[
(a+b)^*(s)
\leq a^*\left(\frac{s}{2}\right)+b^*\left(\frac{s}{2}\right),
\qquad 0<s<|\Omega|.
\] 
Hence avoids the singular difference obstruction. This argument covers the whole
singular interval and does not rely on a background-free stability estimate. Separately,
in a subrange of $1<p<2$, we establish a quantitatively homogeneous difference
estimate by combining Lorentz--Sobolev embedding.\medskip

We now introduce the problems and state the main results. Let $1<p<n$, and
let $H:\R^n\to[0,\infty)$ be an even, uniformly convex norm satisfying the
regularity assumptions stated in \cref{section2}. Set
\[
 \cA_{H}(\xi)=H(\xi)^{p-1}\nabla H(\xi),
 \qquad
 \Delta_p^H u=\operatorname{div}\cA_{H}(\nabla u),
\]
with $\cA_{H}(0)=0$. We consider the Poisson equation
\begin{equation}\label{eq:poisson}
 -\Delta_p^H u +\beta (u)= f
 \qquad\text{in }\Omega
\end{equation}
equipped with homogeneous Dirichlet boundary conditions
\begin{equation}\label{eq:dirichlet}
 u=0
 \qquad\text{on }\partial\Omega.
\end{equation}
Furthermore, we consider in the case $\beta\equiv 0$, the Poisson equation \eqref{eq:poisson} equipped with homogeneous Neumann boundary conditions
\begin{equation}\label{eq:neumann}
 \cA_{H}(\nabla u)\cdot\nu=0
 \qquad\text{on }\partial\Omega.
\end{equation}
In this case, the datum $f\in L^{1}(\Omega)$ needs to satisfy the \emph{compatibility condition}
\begin{displaymath}
 \int_\Omega f\,\dd x=0.
\end{displaymath}

For $f\in L^1(\Omega)$, put
\[
 m=|\Omega|,
 \qquad
 A_f(s)=\int_0^s f^*(r)\,\dd r,
 \qquad
 B_f(s)=\int_s^m f^*(r)r^{-\frac{1}{n'}}\,\dd r.
\]

\begin{theorem}\label{thm:main}
Let $1<p<n$, let $f\in L^1(\Omega)$, and let $u$ be the SOLA of \eqref{eq:poisson}, with either \eqref{eq:dirichlet}
or \eqref{eq:neumann}. Then, for every $0<s<m$,
\begin{equation}\label{eq:main-estimate}
 \big([H(\nabla u)]^*(s)\big)^{p-1}
 \leq C\cR_p f(s),
\end{equation}
where
\begin{equation}\label{eq:Rp}
 \cR_p f(s)=
 \begin{cases}
 s^{-\frac{1}{n'}}A_f(s)+B_f(s),
     & n\ge 2,\qquad 2\leq p<n,\\[2mm]
 s^{-\frac{1}{n'}}\|f\|_{L^1(\Omega)}^{2-p}
 A_f(s)^{p-1}+B_f(s),
     &n\ge 3,\qquad 1<p<2.
 \end{cases}
\end{equation}
The first term in the second line is understood to be zero when $f=0$.
The constant $C$ depends only on $n,p,\lambda,\Lambda$, and the geometric
parameters of $\Omega$; in the Neumann case it may also depend on a relative
anisotropic isoperimetric constant.
\end{theorem}

When $H(\xi)=|\xi|$ and $2\leq p<n$, Theorem \ref{thm:main} reduces to
\eqref{eq:intro-classical}. For a general norm $H$, however, the argument
requires anisotropic perimeter, relative anisotropic isoperimetry, and the
quantitative convexity of $H^2$; the extension is therefore not a mere
change of notation.

The second line of \eqref{eq:Rp} is the genuinely singular part of the
theorem. Since
\[
 A_f(s)\leq\|f\|_{L^1(\Omega)}
 \quad\Longrightarrow\quad
 A_f(s)
 \leq
 \|f\|_{L^1(\Omega)}^{2-p}A_f(s)^{p-1}
 \qquad(1<p<2),
\]
the singular estimate is weaker than the formal analogue of the degenerate
one. Its proof is based on a decomposition of the datum into an $L^1$ part
and an $L^{n,1}$ part. The solution corresponding to the latter has bounded
gradient by \cite{AntoniniCianchi2025} and can therefore serve as a
controlled background in \eqref{eq:intro-weighted}. This device is what
allows \Cref{thm:main} to cover the entire singular interval $1<p<2$.

The weak endpoint contained in this analysis is stated separately, since
it also plays a central role in the stability argument.

\begin{theorem}\label{thm:individual}
Let $1<p<n$ and set
\[
 q=\frac{n(p-1)}{n-1}.
\]
Let $f\in L^1(\Omega)$, and let $u$ be the SOLA of
\eqref{eq:poisson}, with either \eqref{eq:dirichlet} or
\eqref{eq:neumann}. Then
\begin{equation}\label{eq:individual}
 \|H(\nabla u)\|_{L^{q,\infty}(\Omega)}
 \leq
 C\|f\|_{L^1(\Omega)}^{\frac{1}{p-1}}.
\end{equation}
\end{theorem}
Together with the global estimate
\[
 \|H(\nabla u)\|_{L^\infty(\Omega)}
 \leq
 C\|f\|_{L^{n,1}(\Omega)}^{\frac{1}{p-1}}
\]
from \cite{AntoniniCianchi2025}, Theorem \ref{thm:individual} provides the two
endpoint controls behind the rearrangement problem. In the singular regime,
however, they cannot be joined by the same stability estimate used when
$p\geq2$. The correct replacement must depend on the size of both data.

\begin{theorem}\label{thm:singular-stability}
Let
\[
 p_*(n)=\frac{3n-2}{2n-1},
 \qquad
 p_*(n)<p<2,
 \qquad
 q=\frac{n(p-1)}{n-1}.
\]
Let $u$ and $v$ be the SOLA corresponding to
$f,g\in L^1(\Omega)$ and with the same boundary condition. In the
Neumann case, assume that both data have zero mean. Then
\begin{equation}\label{eq:singular-stability}
 \|H(\nabla u-\nabla v)\|_{L^{q,\infty}(\Omega)}
 \leq
 C\|f-g\|_{L^1(\Omega)}
 \big(
 \|f\|_{L^1(\Omega)}+\|g\|_{L^1(\Omega)}
 \big)^{\frac{2-p}{p-1}}.
\end{equation}
\end{theorem}

The right-hand side of \eqref{eq:singular-stability} has exactly the same
homogeneity as the gradient of a solution. The restriction on $p$ is caused
by the Lorentz--Sobolev exponent
\[
 r
 =
 \frac{2q}{q+2-p}
 =
 \frac{2n(p-1)}{n+p-2}.
\]
The proof requires $r>1$, which is equivalent to $p>p_*(n)$. This threshold
belongs to the present method; no failure of
\eqref{eq:singular-stability} at or below $p_*(n)$ is asserted.

We also emphasize that Theorem \ref{thm:main} in the singular range does not rely
on Theorem \ref{thm:singular-stability}. It uses instead the bounded condition described above. Thus the rearrangement estimate holds for every
$1<p<2$, and the free bounded condition for stability theorem is presently obtained only for $p>p_*(n)$.

A central contribution of this paper is that the global gradient
rearrangement estimate covers the entire range $1<p<n$, including the singular interval
$1<p<2$. This is not a consequence of the difference estimate stated above.
In the singular region, the classical coercivity inequality involves the gradients of both solutions are not hold, and the estimation of differential stability cannot be performed for global range due to the influence of the H\'older inequality. We overcome this problem by decomposing the datum into an $L^1$ part and an $L^{n,1}$ part, and comparing the original solution with the solution associated with the latter.

It is therefore important to distinguish the range of the rearrangement theorem from that of the stability theorem.
The rearrangement estimate and the individual weak-type gradient
estimate hold for every $1<p<n$, whereas the stability estimate is established, in the singular regime only for $p_*(n)<p<2$.
The threshold $p_*(n)$ arises from the Lorentz-Sobolev step in the proof of stability.

The global nature of Theorem \ref{thm:main} has several consequences. It yields
the endpoint regularity scale, applies after an exponential change of
variables to equations with a natural $H(\nabla u)^p$ growth term, and can
be tested explicitly on Wulff balls. The Wulff-ball solution also reveals
a feature that is invisible in small-set asymptotics: a finite-tail
representation whose lower limit is exactly $s$ cannot hold uniformly as
$s\uparrow|\Omega|$. A zero extension of $f^*$ or a shifted lower limit is
needed to retain the correct large-$s$ behavior.

The rearrangement estimates yield a systematic passage from
integrability assumptions on the datum to global regularity of
the anisotropic gradient.
We formulate this passage through one-dimensional Hardy
inequalities in rearrangement-invariant spaces.
In the degenerate range $2\leq p<n$, this gives the weak endpoint
estimate for $L^1$ data and the strong estimate
\[
    \|H(\nabla u)\|_{L^{\frac{nq(p-1)}{n-q}}(\Omega)}
    \leq C\|f\|_{L^q(\Omega)}^{\frac{1}{p-1}},
    \qquad 1<q<n.
\]
For $q=n$, the gradient belongs to every finite Lebesgue space,
whereas $q>n$ yields a globally bounded gradient.
We also derive Orlicz and Zygmund estimates in this regime,
including exponential and double-exponential integrability
at the critical scale.

The applications in the singular range $1<p<2$ are stated
separately, reflecting the different structure of the
rearrangement bound.
For $L^1$ data, the same weak endpoint
$H(\nabla u)\in L^{n'(p-1),\infty}(\Omega)$ remains valid
throughout the singular interval.
For $1<q<n$, the singular rearrangement estimate yields
\[
    \|H(\nabla u)\|_{L^{\sigma_{p,q},\infty}(\Omega)}
    \leq C\|f\|_{L^q(\Omega)}^{\frac{1}{p-1}},
    \qquad
    \sigma_{p,q}
    =\frac{p-1}{\frac{1}{n'}-\frac{p-1}{q'}},
    \qquad
    q'=\frac{q}{q-1},
\]
and hence strong $L^\sigma$ estimates for every
$0<\sigma<\sigma_{p,q}$.
At $q=n$, the corresponding weak exponent is
$\frac{n'(p-1)}{2-p}$.
For $q>n$, a bounded gradient follows by combining the global
$L^{n,1}$ endpoint estimate with
$L^q(\Omega)\hookrightarrow L^{n,1}(\Omega)$.
These conclusions distinguish what follows directly from the singular rearrangement bound from what requires the additional
Lorentz endpoint estimate. Moreover, in the degenerate range, we also obtain spatial regularity estimates for almost every time slice of parabolic solutions satisfying the additional relation $\partial_t u=\beta(u)$,
under the hypotheses stated in the corresponding application.

\subsection{Organization of the paper}\label{subsec:organization}

The remaining part of the paper is organized as follows. In
\Cref{section2} we collect the anisotropic structural inequalities,
rearrangement and Lorentz-space tools, level-set formulas, boundary
inequalities, and the definition of SOLA. In
\Cref{section3} we establish the individual weak endpoint estimate.
\Cref{section4} contains the homogeneity obstruction and the difference
estimates in the degenerate and singular regimes. The global rearrangement
theorem is proved in \Cref{section5} through a decomposition of the datum
and a comparison with a solution having bounded gradient. Endpoint
consequences, the natural-gradient-growth application, and the Wulff-ball
model example are discussed in \Cref{section6}. Finally, \Cref{section7} develops applications of the rearrangement estimate to function-space regularity.

\section{Preliminaries}

\subsection{Anisotropic structure and geometry}
\label{section2}

For this section we refer to \cite{CWYY,CY2025ANS} for more details of anisotropic operator (also known as Finsler operator).
Throughout the paper we assume that \(H:\mathbb R^n\to[0,+\infty)\) is a norm
such that
\[
H\in C^2(\mathbb R^n\setminus\{0\}),
\]
and such that the unit ball of \(H\) is strictly convex and has \(C^2\) boundary.
The function $H$ is called anisotropic norm or Finsler norm, and, often, it represents the surface tension density.  In particular,
\begin{equation}\label{eq:H_properties}
H(\xi)>0\quad\text{for }\xi\neq0,\qquad
H(t\xi)=|t|H(\xi)
\quad\text{for every }t\in\mathbb R,\ \xi\in\mathbb R^n.
\end{equation}
Moreover, we have that the anisotropy $H$ satisfies there exist constants $0<\lambda\le\Lambda<\infty$ such that
\begin{equation}\label{eq:ellipticity-H}
 \lambda I
 \le \frac12\nabla_\xi^2H^2(\xi)
 \le \Lambda I
 \quad (\xi\ne 0).
\end{equation} 
The polar norm is denoted by $H^\circ: \R^n \to [0, +\infty)$  and it is defined in the following way 
\begin{equation}
    H^\circ(x) = \sup_{\xi \neq 0} \frac{\langle x, \xi \rangle}{H(\xi)},
\end{equation}
 recalling that 
\begin{equation*}
H(x) = \sup_{\xi \neq 0 } \frac{\left< x , \xi \right>}{H^{\circ}(\xi)}.
\end{equation*}
The Wulff ball, or Wulff shape,  centered at the origin with radius $R$ is defined as:
\[
    \mathcal{W}_R = \{ x \in \R^n : H^\circ(x) < R \}.
\]
 We denote the unit Wulff ball by $\mathcal{W} = \mathcal{W}_1$ and its Lebesgue measure by $\kappa_n = |\mathcal{W}|$. Note that in the Euclidean case where $H(\xi)=|\xi|$, $\mathcal{W}$ is the standard unit ball and $\kappa_n = \omega_n$.
We also recall the following properties for $H$ and $H^{\circ}$:
\begin{equation} \label{prodsc}
H_{\xi}(\xi)\cdot \xi = H(\xi), \;\;\;\; H^{\circ}_{\xi}(\xi)\cdot \xi = H^{\circ}(\xi),
\end{equation} 
\begin{equation} \label{HHcirc}
H(H^{\circ}_{\xi}(\xi))= H^{\circ}(H_{\xi}(\xi))=1 \,\,\,\,\, \forall \xi \in \mathbb{R}^n \setminus \{0\},
\end{equation} 
\begin{equation}\label{Hcircxi}
H^{\circ}(\xi)H_{\xi}(H^{\circ}_{\xi}(\xi))=H(\xi)H^{\circ}_{\xi}(H_{\xi}(\xi)) = \xi \,\,\, \forall \xi \in \mathbb{R}^n \setminus \{0\}.
\end{equation}

 Classical notion of perimeter intended as the $\mathcal{H}^{n-1}$-dimensional measure. This yields the generalized definition of perimeter of a set $E$ of $\Omega$ with anisotropic operator as follows
\begin{align*}\label{2.1.11}
P_{H}(E;\Omega) = \int_{\Omega}|D \chi_{E}(x)|_{H} \mathrm{d}x= \sup\{\int_{E}\mathrm{div}\varrho(x)\;\mathrm{d}x: \varrho \in C^{1}_{c}(\Omega,\mathbb{R}^n), H^{\circ}(\varrho) \leq 1 \}.
\end{align*}\[\]
When $E$ is finite perimeter set, we have
\vskip0.1cm
\begin{align*}\label{2.1.12}
P_{H}(E;\Omega) =   \int_{\Omega}H(\frac{\dd D_{\chi_E}}{\dd |D_{\chi_E}|})\mathrm{d} |D_{\chi_E}|= \int_{\partial^{*}E\cap\Omega}H(\nu)\mathrm{d}\mathcal{H}^{n-1},
\end{align*}\[\]
where $\partial^{*}E$ is the reduced boundary of $E$ and $\nu$ is outer normal (more details about reduced boundary see book \cite{M} section 12.1).

We impose one of the alternatives
\begin{equation}\label{eq:domain}
 \begin{cases}
 \Omega\text{ is bounded and Lipschitz, and }
 \partial\Omega\in W^2L^{n-1,1},&\text{or}\\
 \Omega\text{ is bounded and convex}.&
 \end{cases}
\end{equation}
The first condition means that local graph representations of $\partial\Om$ have weak second derivatives in $L^{n-1,1}$.  Equivalently, the weak second fundamental form belongs locally to this Lorentz space. In the end, for H\"older exponent $p$ and $p'$ satisfying $\frac{1}{p}+\frac{1}{p'}=1$ and $p'=\frac{p}{p-1}$.

\subsection{Rearrangement and Function spaces}

For a measurable function $u$ on $\Om$, let
\[
 \mu_u(t)=|\{x\in\Om:|u(x)|>t\}|,
 \qquad
 u^*(s)=\inf\{t\ge 0:\mu_u(t)\le s\}.
\]
Moreover, we define the increasing rearrangements of $u$ as 
\[
u_*(s)=u^*(|\Omega|-s).
\]
Consequently, for any $p \in [1, \infty]$, we have
\[
    \|u\|_{L^p(\Omega)} = \|u^*\|_{L^p(0, |\Omega|)} = \|u_*\|_{L^p(0, |\Omega|)}.
\]
A key property we will use is the Hardy-Littlewood inequality: for any non-negative functions $f, g \in L^1(\Omega)$,
\[
    \int_\Omega |f(x) g(x)| \, dx \leq \int_0^{|\Omega|} f^*(s) g^*(s) \, ds. 
\]
In particular, choosing $g(x) = \chi_{\{u > t\}}(x)$, we obtain:
\begin{equation} \label{eq:HL_levelsets}
    \int_{\{u > t\}} f(x) \, dx \leq \int_0^{\mu(t)} f^*(s) \, ds.
\end{equation}
Moreover, we have
\[
     \int_0^{|\Omega|} f^*(s) g_*(s) \, ds\leq\int_{\Omega} |f(x)g(x)|dx, 
\]
The maximal rearrangement of \(u\) is defined by
\[
u^{**}(s)
=
\frac1s\int_0^s u^*(r)\,\mathrm dr,
\qquad 0<s<m.
\]
Clearly,
\[
u^*(s)\leq u^{**}(s),
\]
and \(u^{**}\) is non-increasing. 

Let \(0<q<\infty\) and \(0<k\leq\infty\). The Lorentz space
\(L^{q,k}(\Omega)\) consists of all measurable functions
\(u:\Omega\to\mathbb R\) for which the quantity
\[
\|u\|_{L^{q,k}(\Omega)}
=
\left\|s^{1/q}u^*(s)\right\|_{L^k((0,m),\,\mathrm ds/s)}
\]
is finite. More explicitly, when \(0<k<\infty\),
\[
\|u\|_{L^{q,k}(\Omega)}
=
\left(
\int_0^m
\bigl[s^{\frac1q}u^*(s)\bigr]^k
\frac{\mathrm ds}{s}
\right)^{1/k}
=
\left(
\int_0^m
s^{\frac{k}{q}-1}u^*(s)^k\,\mathrm ds
\right)^{1/k},
\]
where
\[
\|u\|_{L^{q,\infty}(\Omega)}
=
\sup_{0<s<m}s^{\frac1q}u^*(s).
\]
We also adopt the convention
\[
L^{\infty,\infty}(\Omega)=L^\infty(\Omega)\qquad L^{q,q}(\Omega)=L^q(\Omega),
\]
with equality of norms for every \(0<q<\infty\).

For \(1<q<\infty\), Hardy's inequality yields
\[
\left\|s^{\frac1q}u^*(s)\right\|_{L^k((0,m),\,\mathrm ds/s)}
\sim
\left\|s^{\frac1q}u^{**}(s)\right\|_{L^k((0,m),\,\mathrm ds/s)},
\]
where the equivalence constants depend only on \(q\) and \(k\).
Accordingly, one may equivalently use
\[
\|u\|_{(q,k)}
=
\left\|s^{\frac1q}u^{**}(s)\right\|_{L^k((0,m),\,\mathrm ds/s)}.
\]
If \(1<q<\infty\) and \(1\leq k\leq\infty\), this expression defines
an equivalent Banach function norm. When \(0<k<1\), the space
\(L^{q,k}(\Omega)\) is, in general, only a quasi-Banach function
space.

The weak Lorentz space \(L^{q,\infty}(\Omega)\), also called the
Marcinkiewicz or weak Lebesgue space, admits the equivalent
distribution-function characterization
\[
\|u\|_{L^{q,\infty}(\Omega)}
=
\sup_{t>0}
t\,\mu_u(t)^{\frac1q}
=
\sup_{0<s<m}s^{\frac1q}u^*(s).
\]
Equivalently, \(u\in L^{q,\infty}(\Omega)\) if and only if there exists
a constant \(C>0\) such that
\[
\mu_u(t)\leq C^q t^{-q}
\qquad\text{for every }t>0.
\]

The Lorentz scale refines the usual Lebesgue scale. More precisely, if
\(0<k_1<k_2\leq\infty\), then
\[
L^{q,k_1}(\Omega)
\hookrightarrow
L^{q,k_2}(\Omega),
\]
and hence
\[
L^{q,1}(\Omega)
\hookrightarrow
L^{q,k}(\Omega)
\hookrightarrow
L^{q,\infty}(\Omega)
\qquad
\text{for every }1\leq k\leq\infty.
\]
Furthermore, since \(|\Omega|<\infty\), if
\[
0<q_2<q_1\leq\infty
\]
and \(0<k_1,k_2\leq\infty\), then
\[
L^{q_1,k_1}(\Omega)
\hookrightarrow
L^{q_2,k_2}(\Omega).
\]
More quantitatively,
\[
\|u\|_{L^{q_2,k_2}(\Omega)}
\leq
C
|\Omega|^{1/q_2-1/q_1}
\|u\|_{L^{q_1,k_1}(\Omega)},
\]
where \(C\) depends only on \(q_1,q_2,k_1\), and \(k_2\).

If \(1<q<\infty\), \(1\leq k\leq\infty\), and
\[
\frac1q+\frac1{q'}=1,
\qquad
\frac1k+\frac1{k'}=1,
\]
with the usual endpoint conventions, then the Lorentz--H\"older
inequality gives
\begin{align}\label{eq:Lorentz--Holder}
\int_\Omega |u(x)v(x)|\,\mathrm dx
\leq
C
\|u\|_{L^{q,k}(\Omega)}
\|v\|_{L^{q',k'}(\Omega)}.
\end{align}
In particular,
\[
L^{q,1}(\Omega)\cdot L^{q',\infty}(\Omega)
\hookrightarrow L^1(\Omega).
\]
This endpoint duality is one of the principal reasons for the
appearance of \(L^{n,1}(\Omega)\) and \(L^{n',\infty}(\Omega)\) in
global gradient estimates.

\subsection{Coercivity and rearrangement tools}

\begin{lemma}\label{lem:monotonicity}
Let $\cA_{H}(\xi)=H(\xi)^{p-1}\nabla H(\xi)$ for $\xi\ne0$.
\begin{enumerate}[label=\textup{(\roman*)}]
\item If $p\ge2$, then
\begin{equation}\label{eq:strong-monotonicity}
 (\cA_{H}(\xi)-\cA_{H}(\eta))\cdot(\xi-\eta)
 \ge c H(\xi-\eta)^p.
\end{equation}
\item If $1<p<2$, then
\begin{equation}\label{eq:weighted-monotonicity}
 (\cA_{H}(\xi)-\cA_{H}(\eta))\cdot(\xi-\eta)
 \ge c[H(\xi)+H(\eta)]^{p-2}H(\xi-\eta)^2.
\end{equation}
\end{enumerate}
The constants depend only on $p,\lambda,\Lambda$.
\end{lemma}

\begin{proof}
Set $F(\xi)=\frac{H(\xi)^p}{p}$.  Euler's identity and direct
differentiation give
\begin{equation}\label{eq:hessian-F}
 \nabla F=\cA_{H},\qquad
 \nabla^2F(\xi)=H(\xi)^{p-2}
 \left[\frac12\nabla_\xi^2H^2(\xi)+(p-2)\nabla H(\xi)\otimes\nabla H(\xi)\right],
\end{equation}
Since $H$ is convex,
\begin{equation}\label{eq:matrix-cauchy}
  \left(\frac12\nabla_\xi^2H^2(\xi)\right)z\cdot z =(H\nabla^2H+\nabla H\otimes\nabla H)z\cdot z \geq (\nabla H(\xi)\cdot z)^2.
\end{equation}
If $p\ge2$, $(p-2)(\nabla H(\xi)\cdot z)^2 \geq 0$. Combining with \eqref{eq:ellipticity-H}, we have
\begin{align*}
 \nabla^2F(\xi)z\cdot z &= H(\xi)^{p-2}
 \left[\frac12\nabla_\xi^2H^2(\xi)+(p-2)\nabla H(\xi)\otimes\nabla H(\xi)\right]z\cdot z\\
 &\ge\lambda H(\xi)^{p-2}|z|^2.
\end{align*}
Integrating on the segment from $\eta$ to $\xi$ and using a norming
functional for $\xi-\eta$ yields \eqref{eq:strong-monotonicity}.

If $1<p<2$, multiplying both sides of \eqref{eq:matrix-cauchy} by $(p-2)H^{p-2}$, then we obtain
\begin{align*}
&(p-2)H^{p-2}\left(\frac12\nabla_\xi^2H^2(\xi)\right)z\cdot z +H^{p-2}\left(\frac12\nabla_\xi^2H^2(\xi)\right)z\cdot z\\
&\leq (p-2)H^{p-2}(\nabla H(\xi)\cdot z)^2 + H^{p-2}\left(\frac12\nabla_\xi^2H^2(\xi)\right)z\cdot z \\
&= H^{p-2}\left((p-1)\nabla H\otimes\nabla H + H\nabla^2H\right)z\cdot z\\
&= \nabla^2F(\zeta)z\cdot z.
\end{align*}
Combining with \eqref{eq:ellipticity-H}, we have
 \begin{align*}
(p-1)\lambda H(\zeta)^{p-2}|z|^2 \leq \nabla^2F(\zeta)z\cdot z
\end{align*}
Let
\[
d=\xi-\eta,\qquad z_t=\eta+td=(1-t)\eta+t\xi,
\qquad 0\le t\le 1.
\]
Since \(A_H=\nabla F\), direct calculation gives
\begin{align*}
\bigl(A_H(\xi)-A_H(\eta)\bigr)\cdot d
&=\int_0^1 \frac{d}{dt}
       \bigl[A_H(z_t)\cdot d\bigr]\,dt\\
&=\int_0^1 D^2F(z_t)d\cdot d\,dt\\
&\ge (p-1)\lambda |d|^2
       \int_0^1 H(z_t)^{p-2}\,dt.
\end{align*}
By convexity and positive one-homogeneity of \(H\),
\[
H(z_t) 
\le (1-t)H(\eta)+tH(\xi)
\le H(\eta)+H(\xi),\qquad 0<t<1.
\]
Since \(p-2<0\), the function \(s\mapsto s^{p-2}\) is decreasing
on \((0,\infty)\). Hence, for almost every \(t\in(0,1)\),
\[
H(z_t)^{p-2}
\ge \bigl[H(\xi)+H(\eta)\bigr]^{p-2}.
\]
Consequently,
\[
\bigl(A_H(\xi)-A_H(\eta)\bigr)\cdot d
\ge c
    \bigl[H(\xi)+H(\eta)\bigr]^{p-2}|d|^2,
\]
which is the proof of
\eqref{eq:weighted-monotonicity}.
\end{proof}

\vspace{1pt}
For \(1<p<2\) and \(L\ge0\), define
\begin{equation}\label{eq:Psi}
  \Psi_L(t)=t^2(L+t)^{p-2}=\frac{t^2}{(L+t)^{2-p}},
  \qquad t\ge0.
\end{equation}
The function \(\Psi_L\) is continuous, strictly increasing, and satisfies \(0\le\Psi_L(t)\le t^p\).

\begin{lemma}\label{lem:bounded-coercivity}
Let \(1<p<2\). If \(H(\eta)\le L\), then
\begin{equation}\label{eq:bounded-coercivity}
  (\cA_H(\xi)-\cA_H(\eta))\cdot(\xi-\eta)
  \ge c\Psi_L\bigl(H(\xi-\eta)\bigr).
\end{equation}
Moreover,
\begin{equation}\label{eq:Psi-inversion-rough}
  t\le C\bigl[\Psi_L(t)^{\frac{1}{p}}+L^{\frac{2-p}{2}}\Psi_L(t)^{\frac{1}{2}}\bigr]
  \qquad t\ge0,
\end{equation}

\end{lemma}

\begin{proof}
Put \(d=H(\xi-\eta)\). By the triangle inequality,
\begin{align*}
  H(\xi)+H(\eta)&= H((\xi-\eta)+\eta)+H(\eta)\\
  &\leq H(\xi-\eta)+2H(\eta)\\
  &\leq d+2L\\
  &\le2(d+L).
\end{align*}
Since \(p-2<0\), \eqref{eq:weighted-monotonicity} gives \eqref{eq:bounded-coercivity}. If \(t\ge L\), then \(L+t\le2t\) and \(\Psi_L(t)^{\frac{1}{p}}\ge ct\). If \(0\le t<L\), then
\[
  L^{\frac{2-p}{2}}\Psi_L(t)^{\frac{1}{2}}
  =t\left(\frac{L}{L+t}\right)^{\frac{2-p}{2}}\ge ct.
\]
This proves \eqref{eq:Psi-inversion-rough}.
\end{proof}

\begin{lemma}\label{lem:distribution-inversion}
Let \(z\ge0\) be measurable and let \(\mu(t)=|\{z>t\}|\). Suppose that, for some \(a,b,\alpha,\beta>0\),
\begin{equation}\label{eq:two-phase-distribution}
  -\mu'(t)\ge c\min\{a\mu(t)^{1+\alpha},\ b\mu(t)^{1+\beta}\}
\end{equation}
for almost every \(t>0\) such that \(\mu(t)>0\). Then
\begin{equation}\label{eq:two-phase-inverted}
  z^*(s)\le C\bigl(a^{-1}s^{-\alpha}+b^{-1}s^{-\beta}\bigr)
  \qquad 0<s<m.
\end{equation}
\end{lemma}

\begin{proof}
Set \(m_0=\mu(0)=|\{z>0\}|\). The conclusion is trivial if
\(s\geq m_0\), since then \(z^*(s)=0\). For \(0<s<m_0\), define
\[
F(r)
=
c\min\{ar^{1+\alpha},br^{1+\beta}\},
\qquad
\Gamma(r)
=
\int_r^{m_0}\frac{d\rho}{F(\rho)}.
\]
Then \(\Gamma\geq0\) is nonnegative, continuous and strictly decreasing on
\((0,m_0]\).

Since \(\mu\) is nonincreasing, it belongs to
\(BV_{\mathrm{loc}}((0,\infty))\). By the BV chain rule and the assumed
differential inequality, the absolutely continuous part of
\(D(\Gamma\circ\mu)\) satisfies
\[
\frac{dD^a(\Gamma\circ\mu)}{dt}
=
\Gamma'(\mu(t))\mu_a'(t)
=
-\frac{\mu_a'(t)}{F(\mu(t))}
\geq1
\]
whenever \(\mu(t)>0\). Moreover, every singular decrease of \(\mu\)
produces a nonnegative contribution to \(D(\Gamma\circ\mu)\), since
\(\Gamma\) is decreasing. Thus
\[
D(\Gamma\circ\mu)\geq dt.
\]
Consequently,
\(\Gamma\circ\mu-\operatorname{id}\) is nondecreasing, and hence
\[
\Gamma(\mu(t))-t
\geq
\Gamma(\mu(0))
=
\Gamma(m_0)
=
0.
\]
Therefore
\[
t\leq\Gamma(\mu(t)).
\]

If \(0<t<z^*(s)\), then the definition of decreasing rearrangement
gives \(\mu(t)>s\). Since \(\Gamma\) is decreasing,
\[
t\leq\Gamma(\mu(t))\leq\Gamma(s).
\]
Letting \(t\uparrow z^*(s)\), we obtain
\[
z^*(s)\leq\Gamma(s).
\]
Finally,
\[
\begin{aligned}
\Gamma(s)
&\leq
\frac1c\int_s^{m_0}
\left(
a^{-1}\rho^{-1-\alpha}
+
b^{-1}\rho^{-1-\beta}
\right)d\rho\\
&\leq
C\left(
a^{-1}s^{-\alpha}
+
b^{-1}s^{-\beta}
\right).
\end{aligned}
\]
This proves the assertion. 

\end{proof}

\begin{lemma}\label{lem:modular-inversion}
Let \(1<p<2\) and \(L,X,t\ge0\). If
\begin{equation}\label{eq:Psi-assumption}
  \Psi_L(t)\le C_0\bigl(X^{p'}+L^{2-p}X^2\bigr),
\end{equation}
then
\begin{equation}\label{eq:Psi-inverted}
  t^{p-1}\le C\bigl[X+X^{p-1}L^{(2-p)(p-1)}\bigr],
\end{equation}
where \(C\) depends only on \(p\) and \(C_0\).
\end{lemma}

\begin{proof}
If \(t\ge L\), then \(\Psi_L(t)\ge ct^p\). Raising \eqref{eq:Psi-assumption} to the power \(\frac{p-1}{p}\) and applying Young inequality gives
\begin{align*}
  t^{p-1} &\le C_0\bigl[X+L^{\frac{(2-p)(p-1)}{p}}X^{\frac{2(p-1)}{p}}\bigr]\\
  &=C_0\bigl[X+X^{\frac{p-1}{p}}(L^{(2-p)}X)^{\frac{1}{p}}\bigr]\\
  &\leq C_0\bigl[X+\frac{p-1}{p}X^{\frac{p-1}{p}}+ \frac{1}{p}L^{(2-p)(p-1)}X^{p-1}  \bigr]\\
  &\leq C\bigl[X+L^{(2-p)(p-1)}X^{p-1} \bigr].
\end{align*}
which is \eqref{eq:Psi-inverted}. If \(0\le t<L\), then \(\Psi_L(t)\ge cL^{p-2}t^2\). Raising the resulting inequality to the power \(\frac{p-1}{2}\) and applying Young inequality gives
\begin{align*}
  t^{p-1} &\leq C_0L^{\frac{(2-p)(p-1)}{2}}\bigl[ X^{p'}+L^{2-p}X^2\bigr]^{\frac{p-1}{2}}\\
  &\leq C_0\bigl[ X^{\frac{p}{2}}L^{\frac{(2-p)(p-1)}{2}}+X^{p-1}L^{(2-p)(p-1)}\bigr]\\
  &\le C\bigl[X+X^{p-1}L^{(2-p)(p-1)}\bigr],
\end{align*}
which proves \eqref{eq:Psi-inverted}. The case \(L=0\) follows directly from \(\Psi_0(t)=t^p\).
\end{proof}

\begin{lemma}\label{lem:rank-weight}
Let \(z\ge0\) be a Sobolev function in \(\Omega\), let \(h\in L^1(\Omega)\) be nonnegative, and assume
\begin{equation}\label{eq:h-critical-zero}
  h=0\qquad\text{a.e. on }\{\nabla z=0\}.
\end{equation}
Let \(I:[0,m]\to[0,\infty)\) be continuous and nondecreasing. Then
\begin{equation}\label{eq:rank-weight-main}
  \int_\Omega I(\mu_z(z(x)))h(x)\,\dd x
  \ge\int_0^m I(r)h^*(r)\,\dd r.
\end{equation}
Consequently, for every \(s\in(0,m)\),
\begin{equation}\label{eq:rank-weight-consequence}
  \int_\Omega I(\mu_z(z(x)))h(x)\,\dd x
  \ge h^*(s)\int_0^s I(r)\,\dd r.
\end{equation}
\end{lemma}

\begin{proof}
Let \(E=\{h>0\}\) and set \(R(x)=\mu_z(z(x))\) on \(E\). We first prove
\begin{equation}\label{eq:rank-subuniform}
  |\{x\in E:R(x)<a\}|\le a
  \qquad(0<a<m).
\end{equation}
Let \(t=z^*(a)\). If \(z(x)<t\), then the definition of the generalized inverse gives \(\mu_z(z(x))>a\). Hence \(R(x)<a\) implies \(z(x)\ge t\). The part where \(z=t\) is negligible in \(E\): Stampacchia's level-set property gives \(\nabla z=0\) almost everywhere on \(\{z=t\}\), and \eqref{eq:h-critical-zero} applies. Since \(\mu_z(t)=|\{z>t\}|\le a\), \eqref{eq:rank-subuniform} follows.

Let \(R_{*,E}\) denote the increasing rearrangement of \(R\) on \(E\). Relation \eqref{eq:rank-subuniform} implies \(R_{*,E}(r)\ge r\) for \(0<r<|E|\). Since \(I\) is continuous and nondecreasing, the increasing rearrangement of \(I(R)\) is at least \(I(r)\). The lower Hardy--Littlewood inequality on \(E\) \cite{BennettSharpley} yields
\[
  \int_E I(R)h\,\dd x
  \ge\int_0^{|E|}[I(R)]_{*,E}(r)h^*(r)\,\dd r
  \ge\int_0^{|E|}I(r)h^*(r)\,\dd r.
\]
Since \(h^*(r)=0\) for \(r>|E|\), this proves \eqref{eq:rank-weight-main}. The monotonicity of \(h^*\) gives \eqref{eq:rank-weight-consequence}.
\end{proof}

\subsection{Interpolation tools}

Finally, for a measurable vector field
\(V:\Omega\to\mathbb R^n\), we use the anisotropic Lorentz notation
\[
L_H^{q,k}(\Omega;\mathbb R^n)
=
\left\{
V:\Omega\to\mathbb R^n:
H(V)\in L^{q,k}(\Omega)
\right\},
\]
endowed with
\[
\|V\|_{L_H^{q,k}(\Omega)}
=
\|H(V)\|_{L^{q,k}(\Omega)}.
\]
Consequently, \(L_H^{q,k}(\Omega;\mathbb R^n)\) coincides, as a set,
with the usual vector-valued Lorentz space
\(L^{q,k}(\Omega;\mathbb R^n)\), and the corresponding quasi-norms are
equivalent. In particular,
\[
\|V\|_{L_H^{q,k}(\Omega)}
\sim
\bigl\||V|\bigr\|_{L^{q,k}(\Omega)},
\]
with equivalence constants depending only on the anisotropy \(H\).

For a compatible couple $(X_0,X_1)$, write
\[
 K(t,h;X_0,X_1)=\inf_{h=h_0+h_1}
 \big(\|h_0\|_{X_0}+t\|h_1\|_{X_1}\big).
\]
If $V:\Om\to\R^n$, the spaces $L^r_H(\Om;\R^n)$ are equipped with the norm
$\|H(V)\|_{L^r}$.  Finite-dimensional norm equivalence and Holmstedt's formulas \cite{BennettSharpley,Holmstedt1970} imply
\begin{align}
 K(t,V;L_H^{q,\infty},L_H^\infty)
 &\sim \sup_{0<r<t^q}r^{\frac1q}[H(V)]^*(r),
 \label{eq:holmstedt-target}\\
 K(\tau,h;L^1,L^{n,1})
 &\sim \int_0^{\tau^{n'}}h^*(r)\,\dd r
 +\tau\int_{\tau^{n'}}^{|\Om|}h^*(r)r^{-\frac{1}{n'}}\,\dd r,
 \label{eq:holmstedt-source}
\end{align}
whenever $0<t^q<|\Om|$ and $0<\tau^{n'}<|\Om|$.  The implicit constants depend only on $n,p$ and the equivalence constants of $H$.

We shall repeatedly use two standard Lorentz H\"older.  If
$0<a,b,c<\infty$, $\frac1c=\frac1a+\frac1b$, then
\begin{equation}\label{eq:Lorentz Holder}
 \|hk\|_{L^{c,\infty}}
 \le C\|h\|_{L^{a,\infty}}\|k\|_{L^{b,\infty}}.
\end{equation}
If $1<r<n$, $0<b\le\infty$, and $r^*=\frac{nr}{n-r}$, then
\begin{equation}\label{eq:lorentz-sobolev}
 \|z\|_{L^{r^*,b}}
 \le C\|H(\nabla z)\|_{L^{r,b}}.
\end{equation}
Here $z\in W_0^{1,r}(\Omega)$ in the Dirichlet case.  In the
Neumann case \eqref{eq:lorentz-sobolev} holds for
$z\in W^{1,r}(\Omega)$ satisfying
$|\supp z|\le \frac{m}{2}$, with a constant depending on the relative
isoperimetric constant.  These inequalities follow from the
rearrangement form of Sobolev's inequality and Lorentz H\"older;
see \cite{BennettSharpley,Mazya2011,ONeil1963}.

\subsection{Distribution functions and SOLA}

For a nonnegative $v$, write
\[
 \mu(t)=|\{v>t\}|,\qquad
 E(t)=\int_{\{v>t\}}H(\nabla v)^p\,\dd x.
\]

\begin{lemma}[Anisotropic distribution inequality]
\label{lem:distribution}
Let $u\in W_0^{1,p}(\Om)$ and assume
that $H$ and domain $\Omega$ satisfy the hypotheses of \Cref{section2}. Then for almost every $t>0$,
\begin{equation}\label{eq:main}
 1 \leq
 C(-\mu'_u(t))^{\frac{p-1}{p}}(\mu_{u}(t))^{-\frac{n-1}{n}}
 \left(-\frac{d}{dt}\int_{\{|u|>t\}}H(\nabla u)^p\;\dd x\right)^{\frac{1}{p}}.
\end{equation}
\end{lemma}

\begin{proof}
Put $v=|u|$.  Since $u\in W_0^{1,p}(\Om)$ and $\Om$ is bounded. Moreover, the evenness and one-homogeneity of $H$ give
\begin{equation}\label{eq:H-absolute-value}
 H(\nabla v)=H(\nabla u)
 \qquad\text{a.e. on }\{u\ne0\}.
\end{equation}

Fix $0<t<T$ and consider the transition layer
\begin{equation*}
 K(t,T)=E_t\setminus E_T=\{x\in\R^n:t<v(x)\le T\},
\end{equation*}
and
\begin{equation*}
 E_{\tau} = \{v>\tau\}.
\end{equation*}
H\"older's inequality with conjugate exponents $p'$ and $p$ yields
\begin{align}\label{holder inequality}
 &\left(
 \int_{K(t,T)}v^{p-1}H(\nabla v)\,\dd x
 \right)^{p'} \notag\\
 &\qquad\le
 \left(\int_{K(t,T)}v^p\,d x\right)
 \left(
 \int_{K(t,T)}H(\nabla v)^p\,d x
 \right)^{\frac{1}{p-1}}.
\end{align}
Dividing by $(T-t)^{p'}$, using anisotropic coarea formula and  Letting
$T\downarrow t$ to the left side gives
\begin{align}\label{anisotropic coarea formula}
\begin{split}
 &\left(\lim_{T\rightarrow t}\frac{1}{T-t}\int_{K(t,T)}v^{p-1}H(\nabla v)\,dx\right)^{p'}\\
 &=\left(\lim_{T\rightarrow t}\frac{1}{T-t}\int_t^T \int_{\partial E_{\tau}}v^{p-1}\chi_{k(t,T)}
 \frac{H(\nabla v)}{|\nabla v|}
 \,d\mathcal{H}^{n-1}\,d\tau\right)^{p'}\\
 &= \left(\lim_{T\rightarrow t}\frac{1}{T-t}\int_t^T \tau^{p-1}\int_{\partial E_{\tau}}\chi_{k(t,T)}
 \frac{H(\nabla v)}{|\nabla v|}
 \,d\mathcal{H}^{n-1}\,d\tau\right)^{p'}\\
 &= \left(\lim_{T\rightarrow t}\frac{1}{T-t}\int_t^T \tau^{p-1}\int_{\partial E_{\tau}}
 H(\nu)
 \,d\mathcal{H}^{n-1}\,d\tau\right)^{p'}\\
 &= \left(\lim_{T\rightarrow t}\frac{1}{T-t}\int_t^T \tau^{p-1}P_{H}(E_{\tau})\,d\tau\right)^{p'}\\
 &= \left(t^{p-1}P_{H}(E_{t})\right)^{p'}.
\end{split}
\end{align}
Dividing by $(T-t)^{p'}$ and  Letting
$T\downarrow t$ to the right side of \eqref{holder inequality}, since $v\le T$ on $K(t,T)$, one has
\begin{align}\label{eq:Lp-right1}
\begin{split}
 \lim_{T\rightarrow t}\frac{1}{T-t}\int_{K(t,T)}v^p\,dx
 &\le \lim_{T\rightarrow t}\frac{1}{T-t}T^p|K(t,T)|\\
 &=\lim_{T\rightarrow t}\frac{1}{T-t}T^p\bigl(\mu_v(t)-\mu_v(T)\bigr)\\
 &= T^p(-\mu^{'}_v(t)).
 \end{split}
\end{align}
In particular
\begin{align}\label{eq:Lp-right2}
\begin{split}
    \left(\lim_{T\rightarrow t}\frac{1}{T-t}\int_{K(t,T)} H(\nabla v)^p\,dx
 \right)^{\frac{1}{p-1}} &= \left(\lim_{T\rightarrow t}\frac{1}{T-t}\int^{T}_{t} \int_{\partial E_{\tau}}\frac{H(\nabla v)^p}{|\nabla v|}\;d\mathcal{H}^{n-1}\,d\tau\right)^{\frac{1}{p-1}}\\
 &=\left(\int_{\partial E_{t}}\frac{H(\nabla v)^p}{|\nabla v|}\;d\mathcal{H}^{n-1}\right)^{\frac{1}{p-1}}.
\end{split}
 \end{align}
for almost every $t>0$. Substituting \eqref{anisotropic coarea formula}--\eqref{eq:Lp-right2} into
\eqref{holder inequality} and we obtain
\begin{align}\label{anisotropic reduction}
\begin{split}
 (nt^{p-1}k_{n}^{\frac{1}{n}}\left(\mu_{v}(t)\right)^{\frac{n-1}{n}})^{p'} &\leq \left(t^{p-1}P_{H}(E_{t})\right)^{p'}\\
 &\leq
 (T^p(-\mu'_v(t)))^{p'}
 \left(\int_{\partial E_{t}}\frac{H(\nabla v)^p}{|\nabla v|}\;d\mathcal{H}^{n-1}\right)^{\frac{1}{p-1}},
\end{split}
\end{align}
i.e.,
\begin{align}
\begin{split}
 1 \leq
 C (-\mu'_v(t))^{\frac{p-1}{p}}(\mu_{v}(t))^{-\frac{n-1}{n}}
 \left(-\frac{d}{dt}\int_{\{v>t\}}H(\nabla v)^p\;\dd x\right)^{\frac{1}{p}}.
\end{split}
\end{align}

\end{proof}

\begin{remark}[Euclidean reduction]\label{rem:euclidean}
If $H(\xi)=|\xi|$, then $P_H(E_t)=\mathcal{H}^{n-1}(\partial^*E_t)$ and
\[
 \frac{H(\nabla u)^p}{|\nabla u|}=|\nabla u|^{p-1}.
\]
Consequently, \eqref{anisotropic reduction} becomes
\[
 \mathcal{H}^{n-1}(\partial^*E_t)^{p'}
 \le
 -\mu_u'(t)
 \left(
 \int_{\{|u|=t\}}|\nabla u|^{p-1}
 \,d\mathcal{H}^{n-1}
 \right)^{\frac{1}{p-1}},
\]
which is precisely Maz'ya's Lemma~2, \cite{Mazya1969}.
\end{remark}

\begin{remark}[anisotropic coarea formula]\label{eq:anisotropic isoperimetric inequality}
For almost every $t>0$, define the anisotropic surface measure
\[
 \dd P_H=H(\nu_{E_t})\,d\mathcal{H}^{n-1}
 \quad\text{on }\partial^*E_t.
\]
The level-set normal identity and the evenness of $H$ imply
\[
 \int_{\{|u|=t\}}
 \frac{H(\nabla u)^p}{|\nabla u|}\,d\mathcal{H}^{n-1}
 =\int_{\partial^*E_t}H(\nabla u)^{p-1}\,d P_H.
\]
Thus \eqref{anisotropic coarea formula} is an entirely intrinsic anisotropic
analogue of the original Euclidean inequality.

(Isoperimetric inequality of anisotropic operator perimeter).\;
If $E$ is a set of finite perimeter in $\mathbb{R}^n$, then
\vskip0.1cm
\begin{align*}
n\kappa_n^{\frac{1}{n}}|E|^{\frac{n-1}{n}} \leq P_{H}(E),
\end{align*}\[\]
and equality holds if and only if, the set $E$ is homothetic to $\mathcal{W}_{H}$.
\end{remark}

If $f$ belongs to the dual of the energy space, weak solutions have their usual variational meaning. For a general $f\in L^1(\Om)$, $\int f\phi \dd x$ need not be well defined. In these cases the most natural way is to proceed by an approximation method as in \cite{dal1997some,DallAglio1996} (see also \cite{boccardo1992nonlinear,boccardo1989non}). The idea consists in fixing the solution as the limit of a sequence of functions which, choose $f_k\in L^\infty(\Om)$ with $f_k\to f$ in $L^1$, impose $\int_\Om f_k=0$ in the Neumann case, solve the regularized problems, and take the almost-everywhere limit supplied by the standard $L^1$ theory for monotone operators.  The resulting gradient is independent of the approximating sequence. We shall refer to
this kind of solution as SOLA, for Solution Obtained as Limit of Approximations
(see \cite{alvino2008nonlinear,boccardo1989non}). Other equivalent notions of solutions have been introduced; we recall the
entropy solution (see \cite{BenilanEtAl1995,boccardo1996existence}) and the renormalized solution (see \cite{DalMasoEtAl1999,lions1965solutions,murat1998equations}). We emphasize that, in any case, they propose the aim to carry out the same limit procedure as that mentioned above. Moreover we point out the existence and uniqueness for SOLAs (see \cite{ACMM,alvino2008nonlinear,prignet1995remarks}). In particular, it turns out
that the weak solution and the approximate solution agree, whenever the former is well defined is the content of the next lemma. The definition of SOLAs are as follows:

\begin{definition}\label{def:approximable}
Let $f\in L^1(\Omega)$. A finite function $u$ is an SOLA of the Dirichlet problem if there are
$f_k\in L^\infty(\Omega)$, $f_k\to f$ in $L^1(\Omega)$, and weak
solutions $u_k\in W_0^{1,p}(\Omega)$ such that, along a subsequence,
\[
 u_k\to u,\qquad \nabla u_k\to\nabla u
 \quad\text{a.e. in }\Omega.
\]
For the Neumann problem we additionally impose
$\int_\Omega f_k=0$, take $u_k\in W^{1,p}(\Omega)$, and normalize
$\med(u_k)=0$.  The generalized gradient is determined by
\[
 \nabla T_t(u)=\chi_{\{|u|<t\}}\nabla u,\qquad
 T_t(u)=\max\{-t,\min\{u,t\}\}.
\]
\end{definition}

\begin{lemma}
\label{lem:L1-energy-obstruction-and-passage}
Let \(\Omega\subset\mathbb{R}^n\) be a bounded domain, let
\(1<p<n\), and let $H$ satisfy the hypotheses of \Cref{section2}. Set
\[
\cA_{H}(\xi)
=
H(\xi)^{p-1}\nabla H(\xi)
\qquad
\text{for }\xi\in\mathbb{R}^n\setminus\{0\}.
\]
Then the following assertions hold.

\begin{enumerate}
\item
An arbitrary function $f\in L^1(\Omega)$ need not define a bounded linear functional on $W_0^{1,p}(\Omega)$. More precisely,
\[
L^1(\Omega)\not\subset
\bigl(W^{1,p}_0(\Omega)\bigr)^*
=
W^{-1,p'}(\Omega).
\]
Consequently, for a general \(L^1\)-integrable, the usual energy
formulation
\begin{equation}
\label{eq:formal-energy-formulation}
\int_\Omega \cA_{H}(\nabla u)\cdot\nabla\varphi\,dx
=
\int_\Omega f\varphi\,dx,
\qquad
\varphi\in W^{1,p}_0(\Omega),
\end{equation}
need not be well defined, and one cannot directly work with an
energy solution \(u\in W^{1,p}_0(\Omega)\).

The analogous problem holds for the Neumann problem, even under
the compatibility condition
\[
\int_\Omega f\,dx=0.
\]

\item
Let \(0<q\leq\infty\), \(\alpha>0\), and \(\beta\geq0\).
Suppose that, for every pair of bounded data
\[
F,G\in L^\infty(\Omega),
\]
the corresponding weak solutions \(u_F\) and \(u_G\), subject to
the same homogeneous boundary condition, satisfy
\begin{equation}
\label{eq:bounded-data-stability-general}
\bigl\|H(\nabla u_F-\nabla u_G)\bigr\|_{L^{q,\infty}(\Omega)}
\leq
C
\|F-G\|_{L^1(\Omega)}^\alpha
\bigl(
\|F\|_{L^1(\Omega)}+\|G\|_{L^1(\Omega)}
\bigr)^\beta,
\end{equation}
where \(C\) is independent of \(F\) and \(G\). In the Neumann case,
\(F\) and \(G\) are assumed to have zero mean, and all solutions are
taken with a fixed common normalization.

Then the same estimate holds for the SOLA
corresponding to arbitrary data \(f,g\in L^1(\Omega)\); namely,
\begin{equation}
\label{eq:L1-data-stability-general}
\bigl\|H(\nabla u_f-\nabla u_g)\bigr\|_{L^{q,\infty}(\Omega)}
\leq
C
\|f-g\|_{L^1(\Omega)}^\alpha
\bigl(
\|f\|_{L^1(\Omega)}+\|g\|_{L^1(\Omega)}
\bigr)^\beta.
\end{equation}
In the Neumann case, the data \(f\) and \(g\) are required to have
zero mean.
\end{enumerate}
\end{lemma}

\begin{proof}
We divide the proof into two steps.

\medskip
\noindent
\emph{Proof of \({\rm(i)}\).}
Fix \(x_0\in\Omega\) and \(r>0\) such that
\[
\overline{B_{2r}(x_0)}\subset\Omega.
\]
Choose a cutoff function
\[
\eta\in C_c^\infty(B_{2r}(x_0)),
\qquad
0\leq\eta\leq1,
\qquad
\eta=1
\quad\text{in }B_r(x_0).
\]
Since \(p<n\), one can choose numbers \(\gamma\) and \(a\) satisfying
\[
0<\gamma<\frac{n-p}{p},
\qquad
n-\gamma\leq a<n.
\]
Define
\[
\varphi(x)
=
\eta(x)|x-x_0|^{-\gamma},
\qquad
f(x)
=
\eta(x)|x-x_0|^{-a}.
\]
Since \(a<n\), one has
\[
f\in L^1(\Omega).
\]
Moreover,
\[
|\nabla\varphi(x)|
\leq
C\bigl(
|x-x_0|^{-\gamma}
+
|x-x_0|^{-\gamma-1}
\bigr),
\]
for almost every \(x\in B_{2r}(x_0)\). The condition
\[
p(\gamma+1)<n,
\]
therefore gives
\[
\varphi\in W^{1,p}_0(\Omega).
\]
On the other hand, since \(a+\gamma\geq n\),
\[
\begin{aligned}
\int_\Omega f\varphi\,dx
&\geq
\int_{B_r(x_0)}
|x-x_0|^{-a-\gamma}\,dx
\\
&=
n\omega_n
\int_0^r
\rho^{n-1-a-\gamma}\,d\rho
=
\infty.
\end{aligned}
\]
Thus, the product \(f\varphi\) need not even be integrable for
\(f\in L^1(\Omega)\) and \(\varphi\in W^{1,p}_0(\Omega)\).

To obtain the failure of boundedness explicitly, define
\[
\varphi_j=T_j(\varphi)=\min\{\varphi,j\}.
\]
Then
\[
\varphi_j\in W^{1,p}_0(\Omega)\cap L^\infty(\Omega)
\]
and, by the Sobolev chain rule,
\[
\nabla\varphi_j
=
\mathbf{1}_{\{\varphi<j\}}\nabla\varphi
\qquad\text{a.e. in }\Omega.
\]
Consequently,
\[
\sup_{j\in\mathbb{N}}
\|\varphi_j\|_{W^{1,p}_0(\Omega)}
\leq
\|\varphi\|_{W^{1,p}_0(\Omega)}
<\infty.
\]
However, the monotone convergence theorem yields
\[
\int_\Omega f\varphi_j\,dx
\longrightarrow
\int_\Omega f\varphi\,dx
=
\infty.
\]
Hence the map
\[
\varphi\longmapsto\int_\Omega f\varphi\,dx
\]
is not a bounded linear functional on \(W^{1,p}_0(\Omega)\). This
proves
\[
L^1(\Omega)\not\subset W^{-1,p'}(\Omega).
\]

The same construction also gives the corresponding obstruction for
the Neumann problem. Indeed, set
\[
\widetilde f
=
f-\frac{1}{|\Omega|}\int_\Omega f\,dx.
\]
Then
\[
\widetilde f\in L^1(\Omega),
\qquad
\int_\Omega\widetilde f\,dx=0.
\]
Since \(\varphi\in L^1(\Omega)\), one has
\[
\sup_{j\in\mathbb{N}}\int_\Omega\varphi_j\,dx
\leq
\int_\Omega\varphi\,dx
<\infty,
\]
and therefore
\[
\int_\Omega\widetilde f\,\varphi_j\,dx
=
\int_\Omega f\varphi_j\,dx
-
\frac{1}{|\Omega|}
\left(\int_\Omega f\,dx\right)
\left(\int_\Omega\varphi_j\,dx\right)
\longrightarrow\infty.
\]
Thus, even the zero-mean condition does not ensure that an arbitrary
\(L^1\)-datum belongs to the dual of the Neumann energy space.

\medskip
\noindent
\emph{Proof of \({\rm(ii)}\).}
Let \(f,g\in L^1(\Omega)\), and choose bounded approximations
\[
F_k,G_k\in L^\infty(\Omega)
\]
such that
\begin{equation}
\label{eq:L1-approximating-data}
F_k\longrightarrow f,
\qquad
G_k\longrightarrow g
\qquad
\text{in }L^1(\Omega).
\end{equation}
In the Neumann case, the approximations may be chosen to satisfy
\[
\int_\Omega F_k\,dx
=
\int_\Omega G_k\,dx
=
0.
\]
Indeed, starting from arbitrary bounded approximations, one may apply
the projection
\[
Ph
=
h-\frac{1}{|\Omega|}\int_\Omega h\,dx.
\]
Since
\[
\|Ph\|_{L^1(\Omega)}
\leq
2\|h\|_{L^1(\Omega)},
\]
the projection \(P\) is bounded on \(L^1(\Omega)\), and it preserves
the required convergence whenever the limiting datum has zero mean.

Let \(u_{F_k}\) and \(u_{G_k}\) denote the corresponding weak solutions. By the defining stability property of SOLA, after passing to a common subsequence,
\[
\nabla u_{F_k}\longrightarrow\nabla u_f,
\qquad
\nabla u_{G_k}\longrightarrow\nabla u_g
\qquad
\text{a.e. in }\Omega.
\]
The continuity of \(H\) therefore implies that
\[
Z_k
:=
H(\nabla u_{F_k}-\nabla u_{G_k})
\longrightarrow
Z
:=
H(\nabla u_f-\nabla u_g)
\qquad
\text{a.e. in }\Omega.
\]

Assume first that \(0<q<\infty\). For every fixed \(\lambda>0\),
almost everywhere convergence gives
\[
\mathbf{1}_{\{Z>\lambda\}}
\leq
\liminf_{k\to\infty}
\mathbf{1}_{\{Z_k>\lambda\}}
\qquad
\text{a.e. in }\Omega.
\]
Fatou's lemma then yields
\[
|\{Z>\lambda\}|
\leq
\liminf_{k\to\infty}
|\{Z_k>\lambda\}|.
\]
Consequently,
\[
\lambda|\{Z>\lambda\}|^{1/q}
\leq
\liminf_{k\to\infty}
\|Z_k\|_{L^{q,\infty}(\Omega)}.
\]
Taking the supremum over \(\lambda>0\), we obtain the lower
semicontinuity property
\begin{equation}
\label{eq:weak-Lorentz-lower-semicontinuity}
\|Z\|_{L^{q,\infty}(\Omega)}
\leq
\liminf_{k\to\infty}
\|Z_k\|_{L^{q,\infty}(\Omega)}.
\end{equation}
If \(q=\infty\), the analogous conclusion follows from the standard
lower semicontinuity of the essential supremum:
\[
\|Z\|_{L^\infty(\Omega)}
\leq
\liminf_{k\to\infty}
\|Z_k\|_{L^\infty(\Omega)}.
\]

Applying \eqref{eq:bounded-data-stability-general} to \(F_k\) and
\(G_k\), and then using
\eqref{eq:weak-Lorentz-lower-semicontinuity}, gives
\[
\begin{aligned}
\bigl\|H(\nabla u_f-\nabla u_g)\bigr\|_{L^{q,\infty}(\Omega)}
&\leq
C\liminf_{k\to\infty}
\|F_k-G_k\|_{L^1(\Omega)}^\alpha
\\
&\quad\times
\bigl(
\|F_k\|_{L^1(\Omega)}
+
\|G_k\|_{L^1(\Omega)}
\bigr)^\beta.
\end{aligned}
\]
By \eqref{eq:L1-approximating-data} and the reverse triangle
inequality,
\[
\|F_k-G_k\|_{L^1(\Omega)}
\longrightarrow
\|f-g\|_{L^1(\Omega)},
\]
and
\[
\|F_k\|_{L^1(\Omega)}
+
\|G_k\|_{L^1(\Omega)}
\longrightarrow
\|f\|_{L^1(\Omega)}
+
\|g\|_{L^1(\Omega)}.
\]
Passing to the limit proves
\eqref{eq:L1-data-stability-general}.
\end{proof}

\section{The individual weak endpoint for all
\texorpdfstring{$1<p<n$}{1<p<n}}\label{section3}

The proof in this section uses only the identity
$\cA_{H}(\nabla u)\cdot\nabla u=H(\nabla u)^p$ and therefore avoids the
difference obstruction in the singular range.

\begin{lemma}\label{lem:weak-estimate-solution}
Let $1<p<n$, let $f\in L^{1}(\Omega)$, and let $u$ be the
SOLA of \eqref{eq:poisson}. Set
\[
    r_{0}=\frac{n(p-1)}{n-p}.
\]
For the Dirichlet problem,
\[
    (u_{\pm})^{*}(s)
    \leq
    C\|f\|_{L^{1}(\Omega)}^{\frac{1}{p-1}}s^{\frac{p-n}{n(p-1)}},
    \qquad 0<s<|\Omega|,
\]
where
\[
    u_{+}=\max\{u,0\},
    \qquad
    u_{-}=\max\{-u,0\}.
\]
For the Neumann problem, the same estimate holds after
the normalization $\operatorname{med}(u)=0$, separately for $u_{+}$
and $u_{-}$. The constant $C$ depends only on $n,p,\lambda,\Lambda$, and the relevant isoperimetric constant.
\end{lemma}

\begin{proof}
We first assume that $f\in L^{\infty}(\Omega)$ and prove the estimate
for the positive part
\[
    v=u_{+}.
\]
Set
\[
    M=\|f\|_{L^{1}(\Omega)},
    \qquad
    \mu(t)=|\{v>t\}|,
    \qquad
    E(t)=\int_{\{v>t\}}H(\nabla v)^{p}\,dx.
\]
If $M=0$, uniqueness in the Dirichlet case and Neumann case, imply that $u=0$. Hence there is nothing to prove. We may therefore assume that $M>0$.

For $h>0$, let
\[
    T_{h}(\tau)=\max\{-h,\min\{\tau,h\}\}.
\]
Then, for $t>0$, testing the equation \eqref{eq:Poisson} with
\begin{displaymath}
    \varphi_{t}=T_{h}\bigl((u-t)_{+}\bigr)
    =
    \begin{cases}
      0    & \text{if $u\le t$,}\\
      u-t  & \text{if $t<u\le t+h$}\\   
      h    & \text{if $u>t+h$,}\\
    \end{cases}
\end{displaymath}
gives that
\begin{align*}
    &\int_{\{t<u<t+h\}}H(\nabla u)^{p}\,dx
    +\int_{\{u>t+h\}}\beta(u)h\,dx\\
    &\qquad\le 
    \int_{\{t<u<t+h\}}H(\nabla u)^{p}\,dx    
    +\int_{\{t<u<t+h\}}\beta(u)(u-t)\,dx
    +\int_{\{u>t+h\}}\beta(u)h\,dx\\
    &=
    \int_{\Omega}
    H(\nabla u)^{p-1}
    \nabla H(\nabla u)\cdot\nabla\varphi_t\,dx
    +\int_{\Omega}\beta(u)\varphi_t\,dx\\
    &=
    \int_{\Omega}f\varphi_t\,dx
    \leq hM,
\end{align*}
where, we used that 
\[
    H(\nabla u)^{p-1}
    \nabla H(\nabla u)\cdot\nabla u
    =
    H(\nabla u)^{p},
\]
the fact that $\beta$ is monotonically increasing, and that $\beta(t)(u-t)> 0$ a.e. on $\{t<u<t+h\}$. Consequently,
\begin{displaymath}
    \frac{E(t)-E(t+h)}{h}+\int_{\{u>t+h\}}\beta(u)\,dx \le\, M
\end{displaymath}
for every $h>0$. Letting $h\downarrow0$, we obtain
\begin{displaymath}
    -E'(t) \le -E'(t)+\int_{\{u>t\}}\beta(u)\,dx \le M
    \qquad\text{for a.e. }t>0.
\end{displaymath}
Combining this estimate with the distribution inequality established
above yields
\[
    1
    \leq
    C[-\mu'(t)]^{\frac{1}{p'}}
    \mu(t)^{-\frac{1}{n'}}
    [-E'(t)]^{\frac{1}{p}}
    \leq
    CM^{\frac{1}{p}}
    [-\mu'(t)]^{\frac{1}{p'}}
    \mu(t)^{-\frac{1}{n'}},
\]
for almost every $t>0$ such that $\mu(t)>0$. Raising both sides to
the power $p'$ gives
\begin{equation}
\label{eq:mu-differential-inequality}
    -\mu'(t)
    \geq
    cM^{-\frac{1}{p-1}}\mu(t)^{\frac{p'}{n'}},
\end{equation}
for almost every $t>0$ with $\mu(t)>0$.

Observe that
\[
\begin{aligned}
    \frac{p'}{n'}
    &=
    \frac{p(n-1)}{n(p-1)}
\\
    &=
    1+\frac{n-p}{n(p-1)}
    =
    1+\frac{1}{r_{0}}.
\end{aligned}
\]
Thus \eqref{eq:mu-differential-inequality} can be rewritten as
\begin{equation}
\label{eq:mu-differential-r0}
    -\mu'(t)
    \geq
    cM^{-\frac{1}{p-1}}
    \mu(t)^{1+\frac{1}{r_{0}}}.
\end{equation}
 On every interval on which $\mu$ is positive, define
\[
    F(t)=\mu(t)^{-\frac{1}{r_{0}}}.
\]
Since $\mu$ is nonincreasing, $F$ is nondecreasing. At almost every
point of differentiability of $\mu$, \eqref{eq:mu-differential-r0}
gives
\[
\begin{aligned}
    F'(t)
    &=
    \frac{1}{r_{0}}
    [-\mu'(t)]
    \mu(t)^{-1-\frac{1}{r_{0}}}
\\
    &\geq
    \frac{c}{r_{0}}
    M^{-\frac{1}{p-1}}.
\end{aligned}
\]
Because $F$ is nondecreasing, its distributional derivative is a
nonnegative Radon measure, and its singular part can only increase
the variation of $F$. Hence, we now integrate \eqref{eq:mu-differential-r0}. for $0<a<t$ such that $\mu(t)>0$,
\[
\begin{aligned}
    F(t)-F(a)
    &\geq
    \int_{a}^{t}F'(\tau)\,d\tau
\\
    &\geq
    \frac{c}{r_{0}}
    M^{-\frac{1}{p-1}}(t-a).
\end{aligned}
\]
Letting $a\downarrow0$ and using $F(a)\geq0$, we obtain
\[
    \mu(t)^{-\frac{1}{r_{0}}}
    \geq
    c_{0}M^{-\frac{1}{p-1}}t.
\]
Equivalently,
\begin{equation}
\label{eq:distribution-weak-estimate}
     \mu(t)
    \leq
    C
    M^{\frac{r_{0}}{p-1}}
    t^{-r_{0}}
    =
    C
    M^{\frac{n}{n-p}}
    t^{-r_{0}}
    \qquad t>0.
\end{equation}
It remains to invert the distribution function. Recall that
the definition of rearrangement and
fix $s\in(0,|\Omega|)$. If $v^{*}(s)=0$, the desired estimate is
immediate. Otherwise, for every $0<t<v^{*}(s)$ satisfying
\[
    \mu(t)>s.
\]
Therefore, by \eqref{eq:distribution-weak-estimate},
\[
    t s^{1/r_{0}}
    <
    t\mu(t)^{\frac{1}{r_{0}}}
    \leq
    CM^{\frac{1}{p-1}}.
\]
Letting $t\uparrow v^{*}(s)$ yields
\[
    v^{*}(s)
    \leq
    CM^{\frac{1}{p-1}}s^{-\frac{1}{r_{0}}}.
\]
we conclude that
\[
    (u_{+})^{*}(s)
    \leq
    C\|f\|_{L^{1}(\Omega)}^{\frac{1}{p-1}}s^{\frac{p-n}{n(p-1)}}.
\]

The estimate for $u_{-}$ follows by applying the same argument to
$-u$. Indeed, since $H$ is even, the vector field
\[
    A(\xi)=H(\xi)^{p-1}\nabla H(\xi)
\]
is odd, and hence $-u$ solves the corresponding equation with function $-f$, whose $L^{1}$ norm is again $M$.

Finally, let $f\in L^{1}(\Omega)$ and let $u$ be the corresponding
SOLA. Choose $f_{j}\in L^{\infty}(\Omega)$ such that
\[
    f_{j}\to f
    \quad\text{in }L^{1}(\Omega).
\]
Let $u_{j}$ be the corresponding weak
solutions. By the definition of SOLA, after
passing to a subsequence,
\[
    u_{j}\to u
    \quad\text{a.e. in }\Omega.
\]
For every $t>0$, Fatou's lemma gives
\[
\begin{aligned}
    |\{u_{+}>t\}|
    &\leq
    \liminf_{j\to\infty}
    |\{(u_{j})_{+}>t\}|
\\
    &\leq
    C
    \liminf_{j\to\infty}
    \|f_{j}\|_{L^{1}(\Omega)}^{\frac{1}{p-1}}
    t^{-r_{0}}
\\
    &=
    C
    \|f\|_{L^{1}(\Omega)}^{\frac{1}{p-1}}
    t^{-r_{0}}.
\end{aligned}
\]
The same estimate holds for $u_{-}$. Inverting these limiting as above and applying definition of SOLA $u$, we have 
\[
    (u_{\pm})^{*}(s)
    \leq
    C\|f\|_{L^{1}(\Omega)}^{\frac{1}{p-1}}s^{\frac{p-n}{n(p-1)}},
\]
and completes the proof.
\end{proof}

\begin{remark}\label{neumann-case-rearrangement}
    In the Neumann case, we normalize the solution by
\[
    \operatorname{med}(u)=0,
\]
with the compatibility condition
\[
    \int_{\Omega}f\,dx=0.
\]
Hence
\[
|\{u>0\}|\leq\frac{|\Omega|}{2},
\qquad
|\{u<0\}|\leq\frac{|\Omega|}{2}.
\]
Consequently, for every \(t>0\),
\[
|\{u_+>t\}|
=
|\{u>t\}|
\leq
\frac{|\Omega|}{2},
\]
and
\[
|\{u_->t\}|
=
|\{u<-t\}|
\leq
\frac{|\Omega|}{2}.
\]
The relative anisotropic isoperimetric inequality therefore yields
\[
|\{u_\pm>t\}|^{\frac{1}{n'}}
\leq
C
P_H(\{u_\pm>t\};\Omega)
\]
for almost every \(t>0\).

Notice that the same argument cannot be used directly
to \(|u|\). Although
\[
\{|u|>t\}
=
\{u>t\}\cup\{u<-t\},
\]
the median normalization does not imply
\[
|\{|u|>t\}|\leq\frac{|\Omega|}{2}.
\]
This is the reason why the distribution estimates must be derived separately for \(u_+\) and \(u_-\).

Since \(-u\) solves the same Neumann problem with function \(-f\). The estimate for \(u_-\) therefore follows from the estimate
for \(u_+\) applied to \(-u\). The two estimates can subsequently be
combined through
\[
|u|^*(s)
\leq
u_+^*\left(\frac{s}{2}\right)+u_-^*\left(\frac{s}{2}\right).
\]
Thus the relative anisotropic isoperimetric inequality, and hence applies separately to $u_{+}$ and $u_{-}$ for Neumann boundary. 

In the proof of interpolation, we set $w = u-v - med(u-v)$. Since a constant translation does not alter the gradient, we have
\[
\nabla w =\nabla u-\nabla v 
\]
and
\[
|\{w>0\}|,|\{w<0\}| \leq \frac{|\Omega|}{2}.
\]
Applying the relative anisotropic isoperimetric inequality to both $w^{+}$ and $w^{-}$ yields similar conclusions. Due to the proof is similar, we will only prove the case for a single solution here.
\end{remark}

\begin{remark}\label{rem:L1-approximable}
In the end of this proof, whenever an estimate is obtained by testing the weak
formulations, the argument is first carried out for bounded data.
The extension to arbitrary \(L^1\)-data and to the corresponding SOLA, this approach is consistent and does not require any additional comment.
\end{remark}

\begin{proof}[Proof of \Cref{thm:individual}]
The argument below should first carried out for bounded data. For a general $f\in L^1(\Omega)$, choose $f_j\in L^\infty(\Omega)$ such that
$f_j\to f$ in $L^1(\Omega)$. To avoid repetition, we write the estimates directly in terms of $f \in L^1(\Omega)$, for more details see the end of Lemma \ref{lem:weak-estimate-solution}.

Let $v=u_+$, $G_+=H(\nabla u)\chi_{\{u>0\}}$, and
$M=\|f\|_{L^1}$.  Fix $0<s<m$ and $\zeta\ge0$ in $L^1(0,s)$, and put
\[
 I(r)=\int_0^{\min\{r,s\}}\zeta(\rho)\,\dd\rho,\qquad
 \Phi(x)=\int_0^{v(x)}I(\mu_v(t))\,\dd t.
\]
By Fubini theorem and Lemma \ref{lem:weak-estimate-solution}, we have
\begin{equation}\label{eq:phi-one}
 \|\Phi\|_{L^\infty}
 \le C M^{\frac{1}{(p-1)}}
 \int_0^s\zeta(\rho)\rho^{\frac{p-n}{n(p-1)}}\,\dd\rho.
\end{equation}
We first notice that
\[
    \Phi(x)=\int_{0}^{v(x)}I(\mu_v(t))\,dt \geq 0,
    \]
the Sobolev chain rule, gives
\[
    \nabla\Phi
    =
    I(\mu_v(v))\nabla v
    =
    I(\mu_v(v))\chi_{\{u>0\}}\nabla u
    \qquad\text{a.e. in }\Omega.
\]
Consequently,
\begin{align*}
    &H(\nabla u)^{p-1}
    \nabla H(\nabla u)\cdot\nabla\Phi
\\
    &\qquad=
    I(\mu_v(v))\chi_{\{u>0\}}
    H(\nabla u)^{p-1}
    \nabla H(\nabla u)\cdot\nabla u
\\
    &\qquad=
    I(\mu_v(v))G_{+}^{p}
    \qquad\text{a.e. in }\Omega.
\end{align*}
Using $\Phi \geq 0$ as a test function, we therefore obtain
\begin{align}
    \int_{\Omega}I(\mu_v(v))G_{+}^{p}\,dx
    &=
    \int_{\Omega}f\Phi\,dx - \int_{\Omega}\beta(u)\Phi\,dx\\
    &\leq
    \|f\|_{L^{1}(\Omega)}
    \|\Phi\|_{L^{\infty}(\Omega)}
\nonumber\\
    &\leq
    CM^{1+\frac{1}{p-1}}
    \int_{0}^{s}\zeta(\rho)\rho^{\alpha}\,d\rho
\nonumber\\
    &=
    CM^{p'}
    \int_{0}^{s}\zeta(\rho)\rho^{\alpha}\,d\rho.
\label{eq:thm12-upper-test}
\end{align}
Indeed, on
$\{u\leq0\}$ one has $G_{+}=0$, whereas on $\{u>0\}$ one has
$\nabla v=\nabla u$. Hence
\[
    G_{+}^{p}=0
    \qquad\text{a.e. on }\{\nabla v=0\}.
\]
Using Lemma \ref{lem:rank-weight} and Fubini theorem, it follows that
\begin{align}\label{eq:thm12-lower-rank}
\begin{split}
\int_{\Omega}I(\mu_v(v))G_{+}^{p}\,dx
    &\geq  
    \int_{0}^{s}(I\circ \mu_v\circ (v^*))_{*}(G_{+}^{p})^{*}(r)\,d\rho\\
    &\geq (G_{+}^{p})^{*}(s)\int_{0}^{s}I(r)\,d\rho\\
    &\geq
    (G_{+}^{p})^{*}(s)
    \int_{0}^{s}\int_{0}^{r}\zeta(\rho)\,\dd \rho \dd r\\
    &=
    [G_{+}^{*}(s)]^{p}
    \int_{0}^{s}\zeta(\rho)(s-\rho)\,d\rho,
    \end{split}
\end{align}
where the last equality we used the fact that the map
$t\mapsto t^{p}$ is continuous and increasing, and hence
\[
    (G_{+}^{p})^{*}(s)=[G_{+}^{*}(s)]^{p}.
\]
Combining \eqref{eq:thm12-upper-test} and
\eqref{eq:thm12-lower-rank}, we obtain
\begin{equation}
\label{eq:thm12-zeta-inequality}
    [G_{+}^{*}(s)]^{p}
    \int_{0}^{s}\zeta(\rho)(s-\rho)\,d\rho
    \leq
    CM^{p'}
    \int_{0}^{s}\zeta(\rho)\rho^{\alpha}\,d\rho.
\end{equation}
Direct calculation gives
\[
\sup\frac{\int_{0}^s\zeta(\rho)(s-\rho)\,\dd\rho}{\int_{0}^s\zeta(\rho)\rho^{\frac{p-n}{n(p-1)}}\,\dd\rho} \ge Cs^{\frac{p(n-1)}{n(p-1)}}.
\]
Taking the supremum over all $\zeta$ on both sides of the equation obtain 
\begin{equation}
\label{eq:thm12-positive-preliminary}
    [G_{+}^{*}(s)]^{p}
    \leq
    CM^{p'}s^{-\frac{p(n-1)}{n(p-1)}}.
\end{equation}
Raising both sides to the power $\frac{1}{p}$, we obtain
\begin{equation}
\label{eq:thm12-positive}
    G_{+}^{*}(s)
    \leq
    CM^{\frac{1}{p-1}}s^{-\frac{1}{q}},
    \qquad 0<s<|\Omega|.
\end{equation}
The same argument applies to the negative part. More precisely, set
\[
    G_{-}
    =
    H(\nabla u)\chi_{\{u<0\}}.
\]
Since the property of $H$, $-u$ solves the corresponding equation with function $-f$, and
\[
    \|-f\|_{L^{1}(\Omega)}=M.
\]
Applying the preceding argument to $-u$ gives
\begin{equation}
\label{eq:thm12-negative}
    G_{-}^{*}(s)
    \leq
    CM^{\frac{1}{p-1}}s^{-\frac{1}{q}},
    \qquad 0<s<|\Omega|.
\end{equation}
By Stampacchia's level-set property,
\[
    \nabla u=0
    \qquad\text{a.e. on }\{u=0\}.
\]
Consequently,
\[
    H(\nabla u)=G_{+}+G_{-}
    \qquad\text{a.e. in }\Omega.
\]
The standard rearrangement inequality
\[
    (a+b)^{*}(s)
    \leq
    a^{*}\left(\frac{s}{2}\right)+b^{*}\left(\frac{s}{2}\right),
    \qquad 0<s<|\Omega|,
\]
together with \eqref{eq:thm12-positive} and
\eqref{eq:thm12-negative}, yields
\begin{align*}
    [H(\nabla u)]^{*}(s)
    &\leq
    G_{+}^{*}\left(\frac{s}{2}\right)+G_{-}^{*}\left(\frac{s}{2}\right)
\\
    &\leq
    CM^{\frac{1}{p-1}}
    \left(\frac{s}{2}\right)^{-1/q}
\\
    &\leq
    \overline{C}M^{\frac{1}{p-1}}s^{-\frac{1}{q}}.
\end{align*}
Therefore,
\begin{equation}
\label{eq:thm12-bounded-data}
    \|H(\nabla u)\|_{L^{q,\infty}(\Omega)}
    =
    \sup_{0<s<|\Omega|}
    s^{\frac{1}{q}}[H(\nabla u)]^{*}(s)
    \leq
    CM^{\frac{1}{p-1}}.
\end{equation}
This argument remains valid when $p<2$, since only the
rearrangement inequality above is used, no triangle
inequality is required. Hence we finish the proof. 
\end{proof}

\begin{remark}[The Neumann case.]
Hence the integration argument of Remark \ref{neumann-case-rearrangement} gives
\[
v^*(r)\leq C M^{\frac{1}{p-1}}
r^{\frac{p-n}{n(p-1)}},\qquad 0<r<m.
\]
Define function $\Phi- \fint_{\Omega}\Phi$ is an test function for Neumann boundary condition with $\int_\Omega f\,dx=0$, which gradient is unchanged. 
\end{remark}

\begin{corollary}\label{cor:infty norm result}
Let $p>1$ and assume that $H$ and domain $\Omega$ satisfy the hypotheses of \Cref{section2}. If
$f\in L^{n,1}(\Omega)$, with zero mean in the Neumann case, then
\begin{equation}\label{eq:lipschitz}
 \|H(\nabla u)\|_{L^\infty(\Omega)}
 \le C\|f\|_{L^{n,1}(\Omega)}^{\frac{1}{p-1}}.
\end{equation}
\end{corollary}
This is the specialization of the global anisotropic
Lipschitz estimates in \cite{AntoniniCianchi2025}.

\section{stability in the singular region}\label{section4}

Numerous papers have studied the stability problem in Euclidean spaces for $p \geq 2$, see \cite{CianchiMazya2014}; however, in this section we first prove that the stability in degenerate cases cannot be directly extended to the case $p <2$. Moreover, in the third section, we prove the stability of the partially singular region.

\subsection{The homogeneity obstruction}
\begin{proposition}[No global super-H\"older stability]
\label{prop:obstruction}
Let $1<p<2$, let $T(f)=\nabla u_f$, and let $\|\cdot\|_X$ be any
homogeneous quasi-norm for which $\|T(f)\|_X>0$ for some function $f$.
There is no constant $C$ such that
\begin{equation}\label{eq:impossible}
 \|T(f)-T(g)\|_X
 \le C\|f-g\|_{L^1}^{\frac{1}{p-1}}
\end{equation}
for all $f,g$.
\end{proposition}

\begin{proof}
Since $1><p<2$, $\frac{1}{p-1}>1$.  The homogeneity of the equation gives
$T(\lambda f)=\lambda^\frac{1}{p-1} T(f)$ for $\lambda>0$.  Choose a nonzero $f$ and put $g=(1-\varepsilon)f$.
If \eqref{eq:impossible} holds, then
\[
 \big[1-(1-\varepsilon)^\frac{1}{p-1}\big]\|T(f)\|_X
 \le C\varepsilon^\frac{1}{p-1}\|f\|_1^\frac{1}{p-1}.
\]
Divide by $\varepsilon$ and let $\varepsilon\downarrow0$.  The
left-hand side converges to
$\frac{1}{p-1}\|T(f)\|_X>0$, however the right-hand side converges to zero, a contradiction.
\end{proof}

\begin{lemma}\label{lem:solution-stability}
Let \(2\leq p<n\) and let \(f,g\in L^{1}(\Omega)\). Let $u$ be the SOLA to \eqref{eq:poisson} with Dirichlet boundary and $v$ is the SOLA to \eqref{eq:poisson} with Dirichlet boundary and replace $f$ by $g$. Then there exists $C>0$ such that
\begin{equation}\label{eq:solution-stability-Dirichlet}
 \|u- v\|_{L^{\frac{n(p-1)}{n-p},\infty}(\Om)}
 \le C\|f-g\|_{L^1(\Om)}^{\frac{1}{p-1}}.
\end{equation}

Assume that $\int_{\Omega} f(x)dx = 0$ and $\int_{\Omega} g(x)dx = 0$. Let $u$ be the solution to \eqref{eq:poisson} with Neumann boundary and $v$ is the weak solution to \eqref{eq:poisson} with Neumann boundary and replace $f$ by $g$. Then
\begin{equation}\label{eq:solution-stability-Neumann}
 \|u- v - \med(u-v)\|_{L^{\frac{n(p-1)}{n-p},\infty}(\Om)}
 \le C\|f-g\|_{L^1(\Om)}^{\frac{1}{p-1}}.
\end{equation}
\end{lemma}

\begin{proof}
We still need to consider the regularity of the approximate solution; however, recalling Remark \ref{rem:L1-approximable}, for  simplicity, we will directly employ the $L^1$ function here.

Fix $t,h>0$ and taking
\[
 \phi_{t,h}(w)=
 \begin{cases}
 0,&w\le t,\\
 w-t,&t<w<t+h,\\
 h,&w\ge t+h,
 \end{cases}
\]
as a test function in the difference of the two weak formulations.  By Lemma \ref{lem:monotonicity},
\begin{align*}
 &\frac{c}{h}\int_{\{t<w<t+h\}}H(\nabla w)^p\,\dd x \\&\leq \int_{\{t<w<t+h\}}\left(H^{p-1}(\nabla u)\nabla H(\nabla u) - H^{p-1}(\nabla v)\nabla H(\nabla v )\right)\cdot(\nabla u-\nabla v)\,\dd x\\
 &\le \int_{\{w>t+h\}}|f-g|\,dx
 +\frac1h\int_{\{t<w<t+h\}}|f-g|(w-t)\,\dd x.
\end{align*}

Letting $h\downarrow 0$ gives, for almost every $t>0$,
\begin{equation}\label{eq:energy-derivative}
 -\frac{\dd}{\dd t}\int_{\{w_+>t\}}H(\nabla w_+)^p\,\dd x\le \int_{\{w>t\}}f-g \dd x \leq \|f-g\|_{L^1(\Omega)}.
\end{equation}
Inserting \eqref{eq:energy-derivative} into Lemma \ref{lem:distribution} and take \(p\)-th root yields
\begin{align*}
 1 &\leq
 C (-\mu'_{w_+}(t))(\mu_{w_+}(t))^{-\frac{p'}{n'}}
 \left(-\frac{d}{dt}\int_{\{w_{+}>t\}}H(\nabla w)^p\;\dd x\right)^{\frac{1}{p-1}}\\
 &\leq  C(-\mu'_{w_+}(t))(\mu_{w_+}(t))^{-\frac{p'}{n'}} \|f-g\|_{L^1(\Omega)}.
\end{align*}
Integration gives
\begin{equation}\label{eq:w+-plus-rearrangement}
t\leq C (\mu_{w_+}(t))^{\frac{p-n}{n(p-1)}} \|f-g\|^{\frac{1}{p-1}}_{L^1(\Omega)}.
\end{equation}
Hence 
\begin{equation}\label{eq:w-plus-rearrangement}
 w_+^*(s)\leq C s^{\frac{p-n}{n(p-1)}} \|f-g\|^{\frac{1}{p-1}}_{L^1(\Omega)}.
\end{equation}

The same argument applies to $w_-$.  Since $w=w_+-w_-$ and weak Lebesgue quasi-norms satisfy a quasi-triangle inequality, \eqref{eq:solution-stability-Dirichlet} follows.
\end{proof}

\subsection{Degenerate stability}

\begin{proposition}
\label{prop:degenerate-stability}
Let $2\le p<n$. If $u, v$ be SOLAS corresponding to $f,g\in L^1(\Om)$, with the same Dirichlet boundary condition.  In the Neumann case assume the compatibility conditions.  Then
\begin{equation}\label{eq:degenerate-stability}
 \|H(\nabla u-\nabla v)\|_{L^{q,\infty}}
 \le C\|f-g\|_{L^1}^{\frac{1}{p-1}},\qquad q=\frac{n(p-1)}{n-1}.
\end{equation}
\end{proposition}

\begin{proof}
  Fix $0<s\le\frac{|\Om|}{2}$, and $\zeta\ge0$ in $L^1(0,s)$. For the Dirichlet problem, set $w=u-v$ and put
\[
 I(r)=\int_0^{\min\{r,s\}}\zeta(\rho)\,\dd\rho,
\]
and
\begin{equation}\label{eq:special-test}
 \Phi(x)=\int_0^{w_+(x)} I(\mu_{w_+}(t))\,\dd t \geq 0.
\end{equation}
The Sobolev chain rule gives
\[
 \nabla\Phi=\chi_{\{w>0\}}I(\mu_{w_+}(w_+))
 (\nabla u-\nabla v).
\]
Using integral by parts we have
\begin{align}\label{divergence-theorem}
&\int_{\Omega}\left(H^{p-1}(\nabla u)\nabla H(\nabla u) - H^{p-1}(\nabla v)\nabla H(\nabla v )\right)\cdot\nabla \Phi\;\dd x \\
&+ \int_{\Omega}(\beta(u)-\beta(v))\Phi\;\dd x=\int_{\Omega}(f-g)\Phi \;\dd x
\end{align}
Inserting definition of $\nabla \Phi$ into \eqref{divergence-theorem} gives the following
\begin{align*}
&\int_{\{w>0\}}I(\mu_{w_+}(w_+)(x))\left(H^{p-1}(\nabla u)\nabla H(\nabla u) - H^{p-1}(\nabla v)\nabla H(\nabla v )\right)\cdot(\nabla u-\nabla v)\;\dd x \\&+ \int_{\{u>v\}}(\beta(u)-\beta(v))\Phi\;\dd x= \int_{\Omega}(f-g)\Phi \;\dd x.
\end{align*}
Moreover, Fubini theorem and Lemma \ref{lem:solution-stability} imply
\begin{align*}
\|\Phi\|_{L^\infty(\Omega)}
&\leq \int_0^\infty
       \int_0^{\min\{\mu_{w_+}(t),s\}}
       \zeta(\rho)\,d\rho\,dt \\
&= \int_0^{(w_+)^*(s)}
       \int_0^s\zeta(\rho)\,d\rho\,dt
   +\int_{(w_+)^*(s)}^\infty
       \int_0^{\mu_{w_+}(t)}
       \zeta(\rho)\,d\rho\,dt \\
&= (w_+)^*(s)\int_0^s\zeta(\rho)\,d\rho
   +\int_0^s
       \int_{(w_+)^*(s)}^{(w_+)^*(\rho)}
       \zeta(\rho)\,dt\,d\rho \\
&= (w_+)^*(s)\int_0^s\zeta(\rho)\,d\rho
   +\int_0^s
       \bigl((w_+)^*(\rho)-(w_+)^*(s)\bigr)
       \zeta(\rho)\,d\rho \\
&= \int_0^s(w_+)^*(\rho)\zeta(\rho)\,d\rho \\
&\leq C\|f-g\|_{L^1(\Omega)}^{\frac{1}{p-1}}
       \int_0^s
       \zeta(\rho)\rho^{\frac{p-n}{n(p-1)}}\,d\rho.
\end{align*}
Since $\beta$ is monotone (increasing) and using Lemma \ref{lem:monotonicity}, Lemma \ref{lem:rank-weight} inequality gives
\begin{align*}
\begin{split}
&\int_{\{w>0\}}I(\mu_{w_+}(w_+)(x))\left(H^{p-1}(\nabla u)\nabla H(\nabla u) - H^{p-1}(\nabla v)\nabla H(\nabla v )\right)\cdot(\nabla u-\nabla v)\;\dd x \\
&+ \int_{\{u>v\}}(\beta(u)-\beta(v))\Phi\;\dd x\\
&\geq \int_0^{\frac{|\Omega|}{2}}\left(I\circ\mu_{w_+}\circ(w_+)^{*}\right)_{*}[H(\nabla w_+)]^*(r)^p\;\dd x\\
&\geq \int_{0}^s I(r)[H(\nabla w_+)]^*(r)^p\;\dd x\\
&\geq [H(\nabla w_+)]^*(s)^p\int_{0}^s\int_{0}^{r}\zeta(\rho)\,\dd\rho\dd r\\
&\geq C[H(\nabla w_+)]^*(s)^p\int_{0}^s\zeta(\rho)(s-\rho)\,\dd\rho.
\end{split}
\end{align*}
 By above inequalities, we obtain
\[
 [H(\nabla w_+)]^*(s)^p \sup\frac{\int_{0}^s\zeta(\rho)(s-\rho)\,\dd\rho}{\int_{0}^s\zeta(\rho)\rho^{\frac{p-n}{n(p-1)}}\,\dd\rho}
 \le C \|f-g\|_{L^1(\Om)}^{\frac{p}{p-1}},
\]
and 
\begin{align}\label{eq: inequality s}
\sup\frac{\int_{0}^s\zeta(\rho)(s-\rho)\,\dd\rho}{\int_{0}^s\zeta(\rho)\rho^{\frac{p-n}{n(p-1)}}\,\dd\rho} \ge Cs^{\frac{p(n-1)}{n(p-1)}}.
\end{align}
This becomes
\begin{equation}\label{eq:plus-gradient-weak}
 [H(\nabla w_+)]^*(s)\le C s^{-\frac{1}{q}}\|f-g\|_{L^1(\Om)}^{\frac{1}{p-1}}.
\end{equation}
The same estimate holds for $w_-$. By Stampacchia's level-set property, $\nabla w=0$ almost everywhere on $\{w=0\}$, hence
\begin{align*}
[H(\nabla u-\nabla v)]^*(s)
&=[H(\nabla w)]^*(s)\\
&\leq [H(\nabla w_+)]^*\left(\frac{s}{2}\right)
     +[H(\nabla w_-)]^*\left(\frac{s}{2}\right)\\
&\leq 2C\left(\frac{s}{2}\right)^{-\frac{1}{q}}
       \|f-g\|_{L^1(\Omega)}^{\frac{1}{p-1}}\\
&\leq Cs^{-\frac{1}{q}}
       \|f-g\|_{L^1(\Omega)}^{\frac{1}{p-1}}.
\end{align*}
For $\frac{|\Omega|}{2}<s<|\Omega|$, the monotonicity of the decreasing rearrangement gives
\[
[H(\nabla w)]^*(s)
\leq [H(\nabla w)]^*\left(\frac{|\Omega|}{2}\right)
\leq C\|f-g\|_{L^1(\Omega)}^{\frac{1}{p-1}}
       \left(\frac{|\Omega|}{2}\right)^{-\frac{1}{q}}.
\]
Since $s<|\Omega|$, we have
$\left(\frac{|\Omega|}{2}\right)^{-\frac{1}{q}}\leq 2^{\frac{1}{q}}s^{-\frac{1}{q}} \leq Cs^{-\frac{1}{q}}$. Hence the same estimate holds for $\frac{|\Omega|}{2}<s<|\Omega|$. And we finish the proof. We mention that for the Neumann case, set
$w=u-v-\operatorname{med}(u-v)$ is similar to prove the results.
\end{proof}

\subsection{Singular stability}
\begin{proof}[Proof of \cref{thm:singular-stability}]
\emph{Step 1.}
We first take $f,g\in L^1(\Omega)$.  Write
\[
 w=u-v,\quad
 M=\|f-g\|_{L^1},\quad N=\|f\|_{L^1}+\|g\|_{L^1},
\]
and, in the Neumann case, replace $w$ by $w-\med(w)$. By \cref{thm:individual} and the weak triangle inequality,
\begin{equation}\label{eq:K-background}
 \|H(\nabla u)+H(\nabla v)\|_{L^{q,\infty}}
 \le C (\|f\|_{L^1}+\|g\|_{L^1})^{\frac{1}{p-1}}=CN^{\frac{1}{p-1}}.
\end{equation}
For \(k>0\), define
\[
T_k(\tau):=\max\{-k,\min\{\tau,k\}\},
\qquad \tau\in\mathbb{R}.
\]
Since \(w_+\geq 0\), one has
\[
T_k(w_+)=\min\{w_+,k\}
=
\begin{cases}
0, & w\leq 0,\\
w, & 0<w<k,\\
k, & w\geq k.
\end{cases}
\]
Moreover, by the Sobolev chain rule,
\[
\nabla T_k(w_+)
=
\chi_{\{0<w<k\}}\nabla w
\qquad\text{a.e. in }\Omega,
\]
and hence, by the positive homogeneity of \(H\),
\[
H\bigl(\nabla T_k(w_+)\bigr)
=
\chi_{\{0<w<k\}}H(\nabla w)
\qquad\text{a.e. in }\Omega.
\]
Testing the difference equation with $T_k(w_+)$ and using
\eqref{eq:weighted-monotonicity} gives
\begin{align}\label{eq:W-strip}
\begin{split}
&\int_{0<w<k}[H(\xi)+H(\eta)]^{p-2}H(\xi-\eta)^2 \;\dd x\\
&\leq \int_{0<w<k}\left(H^{p-1}(\nabla u)\nabla H(\nabla u) - H^{p-1}(\nabla v)\nabla H(\nabla v )\right)\cdot\nabla T_k(w_+)\;\dd x \\
&+ \int_{0<w<k}(\beta(u)-\beta(v))w\;\dd x+ \int_{w>k}(\beta(u)-\beta(v))w\;\dd x\\
&=\int_{\Omega}(f-g)T_k(w_+)\;\dd x\\
&\leq CkM.
\end{split}
\end{align}

\emph{Step 2.}
Define
\begin{equation}\label{eq:r-definition}
 r=\frac{2q}{q+2-p}
   =\frac{2n(p-1)}{n+p-2}.
\end{equation}
The hypothesis $p>p_*(n)$ is equivalent to $r>1$; moreover $r<n$ and
\begin{equation}\label{eq:rstar}
 r^*=\frac{nr}{n-r}
 =\frac{2n(p-1)}{n-p}
\end{equation}
On the set $\{0<w<k\}$, we write
\[
H(\nabla u-\nabla v)
=
\frac{H(\nabla u-\nabla v)}
     {[H(\nabla u)+H(\nabla v)]^{\frac{2-p}{2}}}
[H(\nabla u)+H(\nabla v)]^{\frac{2-p}{2}}.
\]
The Lorentz--H\"older inequality \eqref{eq:Lorentz--Holder} and \eqref{eq:K-background} give
\begin{align}\label{eq:gradient-estimate}
\begin{split}
&\bigl\|H(\nabla u-\nabla v)
       \chi_{\{0<w<k\}}\bigr\|_{L^{r,2}(\Omega)}
\\
&\le C
\left(
 \int_{\{0<w<k\}}
 \frac{H(\nabla u-\nabla v)^2}
      {[H(\nabla u)+H(\nabla v)]^{2-p}}\,dx
\right)^{\frac12}
\left\|
 [H(\nabla u)+H(\nabla v)]^{\frac{2-p}{2}}
\right\|_{L^{\frac{2q}{2-p},\infty}(\Omega)}
\\
&\le C
\bigl(k\|f-g\|_{L^1(\Omega)}\bigr)^{\frac{1}{2}}
\left\|H(\nabla u)+H(\nabla v)
\right\|_{L^{q,\infty}(\Omega)}^{\frac{2-p}{2}}
\\
&\le C
\bigl(\|f-g\|_{L^1(\Omega)}\bigr)^{\frac12}
\bigl(\|f\|_{L^1(\Omega)}+\|g\|_{L^1(\Omega)}
\bigr)^{\frac{2-p}{2(p-1)}}\\
&=CM^{\frac12}N^{\frac{2-p}{2(p-1)}}.
\end{split}
\end{align}
If $\mu(t)=|\{w_+>t\}|$, then
$T_k(w_+)^*(s)\ge k$ for $0<s<\mu(k)$. By the definition of the $L^{r^*,2}$ norm,
\begin{align*}
\|T_k(w_+)\|_{L^{r^*,2}(\Omega)}^2
&=\int_0^{|\Omega|}
  s^{\frac{2}{r^*}}[T_k(w_+)]^*(s)^2\,\frac{ds}{s}\\
&\geq k^2\int_0^{\mu(k)}s^{\frac{2}{r^*}-1}\,ds
 =\frac{r^*}{2}k^2\mu(k)^{\frac{2}{r^*}}.
\end{align*}
On the other hand, the Lorentz--Sobolev inequality \eqref{eq:lorentz-sobolev} and the
preceding gradient estimate \eqref{eq:gradient-estimate} imply
\begin{align*}
\|T_k(w_+)\|_{L^{r^*,2}(\Omega)}
&\leq C\|H(\nabla T_k(w_+))\|_{L^{r,2}(\Omega)}\\
&\leq C
\bigl(k\|f-g\|_{L^1(\Omega)}\bigr)^{\frac12}
\bigl(\|f\|_{L^1(\Omega)}+\|g\|_{L^1(\Omega)}
\bigr)^{\frac{2-p}{2(p-1)}}.
\end{align*}
Combining the last two inequalities and dividing by $k>0$, we obtain
\begin{align*}
k\mu(k)^{\frac{2}{r^*}}&\leq C
\|f-g\|_{L^1(\Omega)}
\bigl(\|f\|_{L^1(\Omega)}+\|g\|_{L^1(\Omega)}
\bigr)^{\frac{2-p}{p-1}}\\
&=CMN^{\frac{2-p}{p-1}}.
\end{align*}
By the definition of the decreasing rearrangement, it follows that
\begin{align}\label{eq:w-estimate}
(w_+)^*(s)\leq CMN^{\frac{2-p}{p-1}}s^{\frac{p-n}{n(p-1)}},
\qquad 0<s<|\Omega|.
\end{align}

\emph{Step 3.}
Fix $0<s<|\Omega|$, take $\zeta\ge0$ in $L^1(0,s)$, and define
\[
 I(r)=\int_0^{\min\{r,s\}}\zeta(\rho)\,\dd\rho,\qquad
 \Phi(x)=\int_0^{w_+(x)}I(\mu(t))\,\dd t.
\]
Applying Fubini theorem and \eqref{eq:w-estimate}, we obtain
\begin{align*}
\|\Phi\|_{L^\infty(\Omega)}
&\leq \int_0^\infty I(\mu(t))\,dt\\
&=\int_0^s \zeta(\rho)
  \left|\{t>0:\mu(t)>\rho\}\right|\,d\rho\\
&=\int_0^s \zeta(\rho)(w_+)^*(\rho)\,d\rho\\
&\leq CMN^{\frac{2-p}{p-1}}\int_0^s
  \zeta(\rho)\rho^{\frac{p-n}{n(p-1)}}\,d\rho.
\end{align*}
By the chain rule,
\[
\nabla\Phi
=I(\mu(w_+))
  (\nabla u-\nabla v)\chi_{\{w>0\}}
\quad\text{a.e. in }\Omega.
\]
Therefore, \ref{lem:monotonicity} and Lemma \ref{lem:rank-weight} yields
\begin{align*}
\int_\Omega
 &\bigl(\cA_H(\nabla u)-\cA_H(\nabla v)\bigr)
 \cdot\nabla\Phi\,dx + \int_{\{u>v\}}(\beta(u)-\beta(v))\Phi\;\dd x\\
&\geq c\int_\Omega I(\mu(w_+))
              [H(\xi)+H(\eta)]^{p-2}H(\xi-\eta)^2\chi_{\{w>0\}}\,dx\\
&\geq c[H(\xi)+H(\eta)]^{p-2}H(\xi-\eta)^2\chi_{\{w>0\}}]^*(s)
       \int_0^s I(r)\,dr.
\end{align*}
By Fubini theorem and the definition of $I$,
\[
\int_0^s I(r)\,dr
=\int_0^s\int_0^r\zeta(\rho)\,d\rho\,dr
=\int_0^s\zeta(\rho)(s-\rho)\,d\rho.
\]
On the other hand, testing the difference equation with $\Phi$
gives
\begin{align*}
&\int_\Omega
 \bigl(\cA_H(\nabla u)-\cA_H(\nabla v)\bigr)
 \cdot\nabla\Phi\,dx + \int_{\{u>v\}}(\beta(u)-\beta(v))\Phi\;\dd x\\
&=\int_\Omega(f-g)\Phi\,dx\leq \|f-g\|_{L^1}\|\Phi\|_{L^\infty(\Omega)}.
\end{align*}
Combining the preceding estimates, we obtain
\begin{align*}
&[H(\xi)+H(\eta)]^{p-2}H(\xi-\eta)^2\chi_{\{w>0\}}]^*(s)\int_0^s\zeta(\rho)(s-\rho)\,d\rho\\
&\leq C\|f-g\|_{L^1}\|\Phi\|_{L^\infty(\Omega)}\\
&\leq CM^2N^{\frac{2-p}{p-1}}\int_0^s
  \zeta(\rho)\rho^{\frac{p-n}{n(p-1)}}\,d\rho.
\end{align*}
As in the proof of Theorem \ref{thm:individual}, we get
\begin{equation}\label{eq:W-weak}
 \|[H(\xi)+H(\eta)]^{p-2}H(\xi-\eta)^2\chi_{\{w>0\}}\|_{L^{q/p,\infty}}
 \le CM^2N^{\frac{2-p}{p-1}}.
\end{equation}

\emph{Step 4.}
By \eqref{eq:K-background}, and \eqref{eq:W-weak},
\begin{align*}
&\bigl\|H(\nabla u-\nabla v)
       \chi_{\{w>0\}}\bigr\|_{L^{q,\infty}(\Omega)}
\\
&\leq C
\left\|
 \frac{H(\nabla u-\nabla v)^2}
      {[H(\nabla u)+H(\nabla v)]^{2-p}}
 \chi_{\{w>0\}}
\right\|_{L^{q/p,\infty}(\Omega)}^{\frac12}
\bigl\|H(\nabla u)+H(\nabla v)\bigr\|_{L^{q,\infty}(\Omega)}^{\frac{2-p}{2}}
\\
&\leq \left(CM^2N^{\frac{2-p}{p-1}}\right)^{\frac12}\left(CN^{\frac{1}{p-1}}\right)^{\frac{2-p}{2}}\\
&\leq CMN^{\frac{2-p}{p-1}}.
\end{align*}
The same argument applied to $w_-$ proves
\[
 \bigl\|H(\nabla u-\nabla v)
      \bigr\|_{L^{q,\infty}(\Omega)}\le  CMN^{\frac{2-p}{p-1}}.
\]
Hence we finish the proof.
\end{proof}

\begin{remark}\label{rem:threshold}
At $p=p_*(n)$, the exponent in \eqref{eq:r-definition} is $r=1$.
The estimate \eqref{eq:gradient-estimate} then gives an $L^{1,2}$ bound, which
does not feed into the strong Lorentz--Sobolev estimate used above.
Thus the present proof genuinely requires $p>p_*(n)$.  No
counterexample is claimed at or below the threshold. Moreover, for the Neumann case, set
$w=u-v-\operatorname{med}(u-v)$ is similar to prove the results.
\end{remark}

\section{Proof of the rearrangement estimate}\label{section5}

\begin{proposition}\label{prop:singular-background}
Let \(1<p<2\), $n \geq 3$. Let \(u,v\) be weak solutions corresponding to \(f,g \in L^{\infty}\), with the same boundary condition, and assume
\[
  L=\|H(\nabla v)\|_{L^\infty(\Omega)}<\infty.
\]
In the Neumann case assume zero mean for both data. Put \(M=\|f-g\|_{L^1(\Omega)}\). Then, for every \(0<s<m\),
\begin{equation} \label{eq:singular-background}
  [H(\nabla u-\nabla v)]^*(s)^{p-1}
  \le C\Bigl[s^{-\frac{1}{n'}}M+s^{-\frac{p-1}{n'}}M^{p-1}L^{(p-1)(2-p)}\Bigr].
\end{equation}
\end{proposition}

\begin{proof}
The case \(M=0\) follows from uniqueness. Define \(w=u-v\), and set
\[
  \mu(t)=|\{w_+>t\}|,
  \qquad
  E(t)=\int_{\{w_+ >t\}}\Psi_L(H(\nabla w_+))\,\dd x.
\]
For \(t>0\) and \(h>0\), define the truncation $T_h:\mathbb{R}\to\mathbb{R}$ by
\[
T_h(\tau):=\max\{-h,\min\{\tau,h\}\}
\]
and
\[
\varphi_{t,h}
   =\frac{1}{h}T_h\bigl((w-t)_+\bigr).
\]
Then \(0\leq \varphi_{t,h}\leq 1\), and
\[
\nabla\varphi_{t,h}
   =\frac{1}{h}\nabla w\chi_{\{t<w<t+h\}}.
\]
Testing the difference equation with \(\varphi_{t,h}\), we obtain
\[
\begin{aligned}
&\frac{1}{h}
\int_{\{t<w_+<t+h\}}
\bigl(\cA_H(\nabla u)-\cA_H(\nabla v)\bigr)
   \cdot(\nabla u-\nabla v)\,dx + \int_{\Omega}(\beta(u)-\beta(v))\varphi_{t,h}\;\dd x\\
&=
\int_{\Omega}(f-g)\varphi_{t,h}\,dx
\leq \|f-g\|_{L^1(\Omega)}
=M.
\end{aligned}
\]
Since \(H(\nabla v)\leq L\) almost everywhere, Lemma \ref{lem:bounded-coercivity} yields
\[
\frac{1}{h}
\int_{\{t<w_+<t+h\}}\Psi_L(H(\nabla w_+))\,dx
\leq CM.
\]
Moreover, by the definition of \(E\), we have
\[
E(t)-E(t+h)
   =\int_{\{t<w_+<t+h\}}\Psi_L(H(\nabla w_+))\,dx.
\]
Dividing both sides by $h$ and letting \(h\downarrow0\) gives
\begin{align}\label{eq:gredient of E}
    -E'(t)\leq CM
\qquad\text{for almost every }t>0.
\end{align}
By Remark \ref{eq:anisotropic isoperimetric inequality}, and Lemma \ref{lem:bounded-coercivity}, we obtain
\begin{align}\label{eq:mu(t)}
\begin{split}
  \mu(t)^{\frac{1}{n'}}
  &\le C\int_{\{w_+=t\}}\frac{H(\nabla w_+)}{|\nabla w_+|}\;\dd \mathcal{H}^{n-1}\\
  &\le C\int_{\{w_+=t\}}\frac{\Psi_L(H(\nabla w_+))^{\frac1p}}{|\nabla w_+|}\;\dd \mathcal{H}^{n-1}
    +CL^{\frac{2-p}{2}}\int_{\{w_+=t\}}\frac{\Psi_L(H(\nabla w_+))^{\frac12}}{|\nabla w_+|}\;\dd \mathcal{H}^{n-1}.
    \end{split}
\end{align}
Applying H\"older's inequality with exponents \(p\) and \(p'\) and Remark \ref{eq:anisotropic isoperimetric inequality}, we obtain
\begin{align*}
\int_{\{w_+=t\}}
\frac{\Psi_L(H(\nabla w_+))^{\frac1p}}{|\nabla w_+|}
\,d\mathcal H^{n-1}
&\le
\left(
\int_{\{w_+=t\}}
\frac{\Psi_L(H(\nabla w_+))}{|\nabla w_+|}
\,d\mathcal H^{n-1}
\right)^{\frac1p} \\
&\quad\times
\left(
\int_{\{w_+=t\}}
\frac{1}{|\nabla w_+|}
\,d\mathcal H^{n-1}
\right)^{\frac{1}{p'}} \\
&\le [-E'(t)]^{\frac1p}[-\mu'(t)]^{\frac{1}{p'}}.
\end{align*}
Similarly, using H\"older's inequality with exponents \(2\) yields
\[
\int_{\{w_+=t\}}
\frac{\Psi_L(H(\nabla w_+))^{\frac12}}{|\nabla w_+|}
\,d\mathcal H^{n-1}
\le [-E'(t)]^{\frac12}[-\mu'(t)]^{\frac12}.
\]
Inserting these into \eqref{eq:mu(t)}, we have
\begin{equation}\label{eq:two-phase-level}
  \mu(t)^{\frac{1}{n'}}
  \le C[-E'(t)]^{\frac1p}[-\mu'(t)]^{\frac{1}{p'}}
  +CL^{\frac{2-p}{2}}[-E'(t)]^{\frac12}[-\mu'(t)]^{\frac12}.
\end{equation}
Combining \eqref{eq:gredient of E} and \eqref{eq:two-phase-level}, at least one term on the right of \eqref{eq:two-phase-level} is at least a fixed multiple of \(\mu(t)^{\frac{1}{n'}}\). Therefore
\begin{equation}\label{eq:two-phase-mu}
  -\mu'(t)
  \ge c\min\Bigl\{M^{-\frac{1}{p-1}}\mu(t)^{\frac{p'}{n'}},\ M^{-1}L^{p-2}\mu(t)^{\frac{2}{n'}}\Bigr\}.
\end{equation} 
Lemma \ref{lem:distribution-inversion} gives
\begin{equation}\label{eq:z-singular}
  z^*(r)
  \le C\bigl[M^{\frac{1}{p-1}}r^{\frac{p-n}{n(p-1)}}+ML^{2-p}r^{\frac{2-n}{n}}\bigr].
\end{equation}
Fix \(s\in(0,m)\), choose \(\zeta\in C_c^\infty((0,s))\), \(\zeta\ge0\), and define \(I\) and \(\Phi\) as Proposition \ref{prop:degenerate-stability}. By Cavalieri and Fubini Theorem, we have
\begin{align}\label{eq:infty norm}
\|\Phi\|_{L^{\infty}} \leq \int_{0}^s z^{*}(r)\zeta(r)\dd r.
\end{align}
By \eqref{eq:z-singular} and \eqref{eq:infty norm},
\begin{equation}\label{eq:Phi-singular}
  \|\Phi\|_{L^\infty}
  \le CM^{\frac{1}{p-1}}\int_0^s\zeta(r)r^{\frac{p-n}{n(p-1)}}\,\dd r
  +CML^{2-p}\int_0^s\zeta(r)r^{\frac{2-n}{n}}\,\dd r.
\end{equation}
Testing the difference equation with \(\Phi\) and using Lemma \ref{lem:bounded-coercivity} gives
\[
  \int_\Omega I(\mu(w_+))\Psi_L(H(\nabla w_+))\,\dd x\le CM\|\Phi\|_{L^\infty}.
\]
Since \(\Psi_L\) is continuous and strictly increasing, \([\Psi_L(H(\nabla w_+))]^*(s)=\Psi_L(H(\nabla w_+)^*(s))\). Hence
\begin{align}\label{eq:Psi-weighted}
  \Psi_L(H(\nabla w_+)^*(s))\int_0^s\zeta(\rho)(s-\rho)\,\dd\rho
  &\le CM^{p'}\int_0^s\zeta(\rho)\rho^{\frac{p-n}{n(p-1)}}\,\dd\rho \notag\\
  &\quad+CM^2L^{2-p}\int_0^s\zeta(\rho)\rho^{\frac{2-n}{n}}\,\dd\rho.
\end{align}
 Using similar inequality \eqref{eq: inequality s}, we obtain
\begin{equation}\label{eq:Psi-pointwise}
  \Psi_L(H(\nabla w_+)^*(s))
  \le C\bigl[M^{p'}s^{-\frac{p'}{n'}}+M^2L^{2-p}s^{-\frac{2}{n'}}\bigr].
\end{equation}
Set \(X=Ms^{-\frac{1}{n'}}\). Then \eqref{eq:Psi-pointwise} reads
\[
  \Psi_L(H(\nabla w_+)^*(s))\le C\bigl[X^{p'}+L^{2-p}X^2\bigr].
\]
By Lemma \ref{lem:modular-inversion},
\begin{equation}\label{eq:D-rearrangement}
  H(\nabla w_+)^*(s)^{p-1}
  \le C\bigl[Ms^{-\frac{1}{n'}}+M^{p-1}s^{-\frac{p-1}{n'}}L^{(p-1)(2-p)}\bigr].
\end{equation}
For \(w_-\), use the similar method. Combining the positive and negative parts and using the subadditivity of \(t\mapsto t^{p-1}\), we get 
\begin{align*}
\begin{split}
  H(\nabla w)^*(s)^{p-1} &\leq H(\nabla w_+)^*\left(\frac{s}{2}\right)^{p-1} + H(\nabla w_-)^*\left(\frac{s}{2}\right)^{p-1}
  \\&\leq C\bigl[Ms^{-\frac{1}{n'}}+M^{p-1}s^{-\frac{p-1}{n'}}L^{(p-1)(2-p)}\bigr],
  \end{split}
\end{align*}
which proofs \eqref{eq:singular-background}.
\end{proof}

\begin{proof}[Proof of Theorem \ref{thm:main}]
Assume first that \(f\in L^\infty(\Omega)\). Fix \(s\in(0,m)\) and put
\[
  \tau=f^*(s),
  \qquad
  f_0=f- \max\{-\tau,\min\{f,\tau\}\},
  \qquad
  f_1=\max\{-\tau,\min\{f,\tau\}\}.
\]
Then \(f=f_0+f_1\), since \(f^*\) is nonincreasing
and \(\tau=f^*(s)\), one has
\[
(f^*(r)-\tau)_+=0\quad\text{for a.e. }r>s.
\]
Consequently,
\begin{align}\label{eq:lorentz norm of f0}
\begin{split}
\|f_0\|_{L^1(\Omega)}
&=\int_0^m(f^*(r)-\tau)_+\,dr  \\
&=\int_0^s(f^*(r)-\tau)\,dr
 =A_f(s)-s\tau
 \leq A_f(s).
 \end{split}
\end{align}
On the other hand, \(|f_1|\leq |f|\), and hence
\[
\|f_1\|_{L^1(\Omega)}
\leq \|f\|_{L^1(\Omega)}.
\]
Since
\[
f_1^*(r)\leq \tau \;\;\text{on}\;(0,s)
\qquad\text{and}\qquad
f_1^*(r)\leq f^*(r) \;\;\text{on}\;(s,m),
\]
the definition of the Lorentz norm yields
\[
\begin{aligned}
\|f_1\|_{L^{n,1}(\Omega)}
&\leq
C\int_0^m f_1^*(r)r^{-\frac{1}{n'}}\,dr \\
&\leq
C\tau\int_0^s r^{-\frac{1}{n'}}\,dr
+C\int_s^m f^*(r)r^{-\frac{1}{n'}}\,dr \\
&=
Cn\tau s^{\frac1n}+CB_f(s).
\end{aligned}
\]
Finally, the monotonicity of \(f^*\) gives
\[
s\tau=s f^*(s)
\leq\int_0^s f^*(r)\,dr=A_f(s).
\]
Since \(s^{\frac1n}=s^{-\frac{1}{n'}}s\), we conclude that
\begin{align}\label{eq:lorentz norm of f1}
\|f_1\|_{L^{n,1}(\Omega)}
\leq
C\left[ns^{-\frac{1}{n'}}A_f(s)+B_f(s)\right].
\end{align}

Let \(u\) and \(u_1\) denote the SOLAs associated with
the data \(f\) and \(f_1\), respectively, with the same boundary condition. Since \(H\) is a norm, the triangle inequality gives
\[
H(\nabla u)
\leq H(\nabla u-\nabla u_1)+H(\nabla u_1)
\qquad\text{a.e. in }\Omega.
\]
Recall that, for any nonnegative measurable functions \(a\) and \(b\),
\[
(a+b)^*(s)
\leq a^*\left(\frac{s}{2}\right)+b^*\left(\frac{s}{2}\right),
\qquad 0<s<|\Omega|.
\]
Applying this rearrangement inequality to the preceding estimate yields
\begin{equation}\label{eq:inequality of two parts}
[H(\nabla u)]^*(s)
\leq
[H(\nabla u-\nabla u_1)]^*\left(\frac{s}{2}\right)
+
[H(\nabla u_1)]^*\left(\frac{s}{2}\right).
\end{equation}
Indeed, by Corollary \ref{cor:infty norm result} and \eqref{eq:lorentz norm of f1},
\begin{equation}
\begin{aligned}
\|H(\nabla u_1)\|_{L^\infty(\Omega)}
&\leq
C\|f_1\|_{L^{n,1}(\Omega)}^{\frac{1}{p-1}}\\
&\leq
C\left(
n s^{-\frac{1}{n'}}A_f(s)+B_f(s)
\right)^{\frac{1}{p-1}}.
\end{aligned}
\label{eq:regular-gradient-bound}
\end{equation}
It follows that
\begin{equation}
[H(\nabla u_1)]^*\left(\frac{s}{2}\right)
\leq
C(n)\left(
s^{-\frac{1}{n'}}A_f(s)+B_f(s)
\right)^{\frac{1}{p-1}}.
\label{eq:regular-rearrangement-bound}
\end{equation}

We now distinguish between the degenerate and singular region.

\medskip
\noindent
\emph{Case 1: $2\leq p<n$.}

By Proposition
\ref{prop:degenerate-stability} and \eqref{eq:lorentz norm of f0}, we have
\begin{align}\label{eq: difference rearrangement}
  [H(\nabla u-\nabla u_1)]^*\left(\frac{s}{2}\right)^{p-1}
  \le Cs^{-\frac{p-1}{q}}A_f(s).
\end{align}
Since \(-\frac{p-1}{q}=-\frac{1}{n'}\), using \eqref{eq:inequality of two parts}, \eqref{eq: difference rearrangement}, and the fact that \((a+b)^{p-1}\le C(a^{p-1}+b^{p-1})\), we have
\[
  [H(\nabla u)]^*(s)^{p-1}
  \le C\bigl[s^{-\frac{1}{n'}}A_f(s)+B_f(s)\bigr].
\]
This is the first line of \eqref{eq:Rp}.

\medskip
\noindent
\emph{Case 2: $1< p<2$.}

Put
\[
  N=\|f\|_{L^1},
  \qquad
  M=\|f_0\|_{L^1}.
\]
Apply Proposition \ref{prop:singular-background} with \(v=u_1\). It gives
\begin{equation}\label{eq:singular-after-split}
  [H(\nabla u-\nabla u_1)]^*\left(\frac{s}{2}\right)^{p-1}
  \le C\Bigl[s^{-\frac{1}{n'}}\|f_0\|_{L^1}+s^{-\frac{p-1}{n'}}\|f_0\|_{L^1}^{p-1}L^{(p-1)(2-p)}\Bigr].
\end{equation}
The first term, since \eqref{eq:lorentz norm of f0} and the definition of $A_f(s)$, we have \(\|f_0\|_{L^1}\le C_1A_f(s)\le C_2\|f\|_{L^1}\) and
\begin{equation}\label{eq:absorb-linear}
  A_f(s)\leq C\|f\|_{L^1}^{2-p}A_f(s)^{p-1}.
\end{equation}
Hence, we have
\[
s^{-\frac{1}{n'}}\|f_0\|_{L^1} \leq C_1s^{-\frac{1}{n'}}A_{f}(s)\leq Cs^{-\frac{1}{n'}}\|f\|_{L^1}^{2-p}A_f(s)^{p-1}.
\]
We notice that the definition of $A_f(s)$ and $B_f(s)$. Since \(r^{-\frac{1}{n'}}\leq s^{-\frac{1}{n'}}\) for \(r\geq s\), we have
\begin{align}
\begin{split}
B_f(s)
&=\int_s^{|\Omega|}f^*(r)r^{-\frac{1}{n'}}\,dr\\
&\leq s^{-\frac{1}{n'}}\int_s^{|\Omega|}f^*(r)\,dr\\
&=s^{-\frac{1}{n'}}\bigl(\|f\|_{L^1(\Omega)}-A_f(s)\bigr).
\end{split}
\end{align}
Consequently, using again \(\|f_0\|_{L^1}\le C_1A_f(s)\le C_2\|f\|_{L^1}\) and \eqref{eq:lorentz norm of f1}, we obtain
\begin{align}\label{eq:last sum of f1 norm}
\begin{split}
\|f_1\|_{L^{n,1}(\Omega)}
&\leq C\left(ns^{-\frac{1}{n'}}A_f(s)+B_f(s)\right)\\
&\leq Cs^{-\frac{1}{n'}}
   \left(\|f\|_{L^1}+(n-1)A_f(s)\right)\\
&\leq Cn\|f\|_{L^1}s^{-\frac{1}{n'}}.
\end{split}
\end{align}
For the second term of \eqref{eq:singular-after-split}, applying Corollary \ref{cor:infty norm result} and \eqref{eq:last sum of f1 norm} we obtain
\begin{align}\label{eq:second-singular-term}
  s^{-\frac{p-1}{n'}}\|f_0\|_{L^1}^{p-1}L^{(p-1)(2-p)}
  &\le Cs^{-\frac{p-1}{n'}}A_f(s)^{p-1}\bigl(\|f\|_{L^1}s^{-\frac{1}{n'}}\bigr)^{2-p}\notag\\
  &=Cs^{-\frac{1}{n'}}\|f\|_{L^1}^{2-p}A_f(s)^{p-1}.
\end{align}
Since \(0<p-1<1\), the map \(t\mapsto t^{p-1}\) is
subadditive on \([0,\infty)\). Hence, combining with inequalities above, we obtain
\begin{align*}
[H(\nabla u)]^*(s)^{p-1}
&\leq
[H(\nabla u-\nabla u_1)]^*\left(\frac{s}{2}\right)^{p-1}
+[H(\nabla u_1)]^*\left(\frac{s}{2}\right)^{p-1}\\
&\leq
C\Bigl[s^{-\frac{1}{n'}}\|f_0\|_{L^1}+s^{-\frac{p-1}{n'}}\|f_0\|_{L^1}^{p-1}L^{(p-1)(2-p)}\Bigr]
+\|H(\nabla u_1)\|_{L^\infty(\Omega)}^{p-1}\\
&\leq C s^{-\frac{1}{n'}}\|f\|_{L^1}^{2-p}A_f(s)^{p-1} + 
  \|H(\nabla u_1)\|_{L^\infty(\Omega)}^{p-1}
\end{align*}
Corollary \ref{cor:infty norm result}, \eqref{eq:lorentz norm of f1} and \eqref{eq:absorb-linear} give
\begin{align}
\begin{split}
\|H(\nabla u_1)\|_{L^\infty(\Omega)}^{p-1}
&\leq 
C \|f_1\|_{L^{n,1}(\Omega)}\\
&\leq C\left[s^{-\frac{1}{n'}}A_f(s)+B_f(s)\right]\\
&= C\left[ s^{-\frac{1}{n'}}N^{2-p}A_f(s)^{p-1} +B_f(s)\right].
\end{split}
\end{align}
Therefore,
\[
[H(\nabla u)]^*(s)^{p-1}
\leq
C\left[
s^{-\frac{1}{n'}}N^{2-p}A_f(s)^{p-1}
+B_f(s)
\right].
\]
Since we choose $f\in L^{\infty}(\Omega)$ first, then let take $f\in L^{1}(\Omega)$ and let $u$ be the corresponding SOLA. Choose $f_{j}\in L^{\infty}(\Omega)$ such that
\[
    f_{j}\to f
    \quad\text{in }L^{1}(\Omega).
\]
Let $u_{j}$ be the corresponding weak
solutions. Recalling Remark \ref{rem:L1-approximable}, which finishes the proof the Theorem \ref{thm:main}.
\end{proof}

\begin{remark}
Fix \(0<s<m =|\Omega|\). In the Neumann case, the
compatibility condition has zero mean. Although
\[
\int_\Omega f\,dx=0
\qquad\text{and}\qquad
f=f_0+f_1,
\]
the functions \(f_0\) and \(f_1\) need not have zero mean separately.
We therefore define
\[
Ph=h-\frac{1}{m}\int_\Omega h\,dx
\]
and replace \(f_0,f_1\) by \(Pf_0,Pf_1\). Then
\[
\int_\Omega Pf_i\,dx=0,\qquad i=0,1,
\]
and, since \(Pf=f\),
\[
Pf_0+Pf_1=P(f_0+f_1)=Pf=f.
\]
Moreover,
\[
\|Pf_0\|_{L^1(\Omega)}
\leq
\|f_0\|_{L^1(\Omega)}
+\left|\int_\Omega f_0\,dx\right|
\leq 2A_f(s).
\]
Since \(\int_\Omega f\,dx=0\), one also has
\[
\left|\int_\Omega f_1\,dx\right|
=
\left|\int_\Omega f_0\,dx\right|
\leq A_f(s).
\]
Using
\[
\|1\|_{L^{n,1}(\Omega)}=nm^{1/n}
\]
and \(0<s<m\), which implies
\[
m^{-\frac{1}{n'}}\leq s^{-\frac{1}{n'}},
\]
we obtain
\begin{align*}
\|Pf_1\|_{L^{n,1}(\Omega)}
&\leq
\|f_1\|_{L^{n,1}(\Omega)}
+\frac{1}{m}
 \left|\int_\Omega f_1\,dx\right|
 \|1\|_{L^{n,1}(\Omega)}\\
&\leq
C\left[s^{-\frac{1}{n'}}A_f(s)+B_f(s)\right]
+nm^{-\frac{1}{n'}}A_f(s)\\
&\leq
C\left[s^{-\frac{1}{n'}}A_f(s)+B_f(s)\right].
\end{align*}
Thus the projection preserves both the decomposition and the required
estimates while making each component admissible as a Neumann datum.
Henceforth, in the Neumann case, we continue to write \(f_0,f_1\) for
the projected functions \(Pf_0,Pf_1\).
\end{remark}

\section{Endpoint consequences and model examples}\label{section6}

\subsection{Endpoint formulations}

\begin{corollary}\label{cor:endpoints}
Under the hypotheses of \cref{thm:main},
\[
 L^1(\Omega)\longrightarrow L^{q,\infty}(\Omega)
 \quad\text{and}\quad
 L^{n,1}(\Omega)\longrightarrow L^\infty(\Omega).
\]
 The first assertion
in fact holds for all $1<p<n$.
\end{corollary}

In the degenerate range, the usual finite- and infinite-tail
representations remain available.  If $f^*$ is extended by zero to
$(m,\infty)$ and
\[
 \widetilde f^{**}(r)=\frac1r\int_0^r\widetilde f^*(t)\,\dd t,
\]
then Fubini theorem gives the exact identity
\begin{equation}\label{eq:infinite-tail}
 \int_s^\infty\widetilde f^{**}(r)r^{-\frac{1}{n'}}\,\dd r
 =n'\left[s^{-\frac{1}{n'}}A_f(s)+B_f(s)\right].
\end{equation}
For $0<s\le m/2$, the same expression is bounded by
$C\int_s^m f^{**}(r)r^{-\frac{1}{n'}}\,\dd r$.  Near $s=m$, a half-shifted
lower limit or the zero-extended formula
\eqref{eq:infinite-tail} is essential.

\subsection{A natural gradient-growth balance}

\begin{corollary}\label{prop:natural-growth}
Let $p_*(n)<p<n$ and let $u$ be a bounded weak solution, with either boundary condition, of
\begin{equation}\label{eq:natural-equation}
 -\Delta_p^H u=g(x,u,\nabla u)\quad\text{in }\Omega.
\end{equation}
Assume that, for some $\kappa\in\R$ and $F\in L^1(\Omega)$,
\begin{equation}\label{eq:natural-balance}
 |g(x,u,\nabla u)-\kappa H(\nabla u)^p|
 \le F(x)\quad\text{a.e.} \text{in }\Omega.
\end{equation}
 Then
\begin{equation}\label{eq:natural-result}
 [H(\nabla u)]^*(s)^{p-1}
 \le C
 \exp\!\big(\esssup_\Omega\kappa u-\essinf_\Omega\kappa u\big)
 \cR_pF(s).
\end{equation}
\end{corollary}

\begin{proof}
Define
\[
 \psi(t)=\int_0^t
 \exp\!\left(\frac{\kappa\tau}{p-1}\right)\dd\tau,
 \qquad v=\psi(u).
\]
The $(p-1)$-homogeneity of $\cA$ gives
\[
 \cA(\nabla v)=e^{\kappa u}\cA(\nabla u)
\]
and hence
\begin{equation}\label{eq:transformed}
 -\Delta_p^H v=e^{\kappa u}
 \big(g(x,u,\nabla u)-\kappa H(\nabla u)^p\big).
\end{equation}
The boundary condition is preserved. And the transformed datum has
zero mean in the Neumann case.

Let $h$ denote the right-hand side of \eqref{eq:transformed}.  From
\eqref{eq:natural-balance},
$|h|\le e^{\esssup\kappa u}F$.  Since both lines of
\eqref{eq:Rp} are homogeneous of degree one and monotone yields,
\[
 \cR_ph(s)\le e^{\esssup\kappa u}\cR_pF(s).
\]
Also,
\[
 H(\nabla v)^{p-1}=e^{\kappa u}H(\nabla u)^{p-1}.
\]
Apply \cref{thm:main} to $v$ and use
$e^{-\kappa u}\le e^{-\essinf\kappa u}$ to obtain
\eqref{eq:natural-result}.
\end{proof}

\subsection{Wulff-ball solution}

\begin{corollary}\label{prop:wulff}
Let $\Omega=\Wulff_R$ and $f\equiv1$.  Then
\begin{equation}\label{eq:wulff-solution}
 u(x)=\frac{p-1}{p}\,n^{-\frac{1}{p-1}}
 \left(R^{p'}-H^\circ(x)^{p'}\right)
\end{equation}
solves the homogeneous Dirichlet problem, and
\begin{equation}\label{eq:wulff-rearrangement}
 [H(\nabla u)]^*(s)^{p-1}
 =\frac1n\left(R^n-\frac{s}{\kappa_n}\right)^{\frac1n},
 \qquad 0<s<\kappa_n R^n.
\end{equation}
\end{corollary}

\begin{proof}
Set $r=H^\circ(x)$.  By property of $H$, we have
\[
 H(\nabla u)=n^{-\frac{1}{p-1}}r^{\frac{1}{p-1}},\qquad
 \cA(\nabla u)=-\frac{x}{n}.
\]
Thus $-\operatorname{div}\cA(\nabla u)=1$, and the boundary value
vanishes at $r=R$.  Since
$|\Wulff_\rho|=\kappa_n\rho^n$ and $H(\nabla u)$ is increasing in
$r$, inversion of its distribution function gives
\eqref{eq:wulff-rearrangement}.
\end{proof}

\begin{remark}\label{rem:large-s}
For $f\equiv1$ and $m=\kappa_n R^n$,
\[
 \int_s^m f^{**}(r)r^{-\frac{1}{n'}}\,\dd r
 =n(m^{\frac1n}-s^{\frac1n})\sim m^{-\frac{1}{n'}}(m-s)
\]
as $s\uparrow m$, whereas the left-hand side of
\eqref{eq:wulff-rearrangement} has order $(m-s)^{\frac1n}$. Hence it is not
valid on the entire interval.  Formula \eqref{eq:main} has the
correct large-$s$ structure.  In particular, in the singular range
its first term does not spuriously vanish as $s\rightarrow m$.
\end{remark}

\section{Applications to function-space estimates}
\label{section7}

\begin{lemma}\label{app:lem:X-into-L1}
Let $\Omega$ be a finite measurable set and let $X(\Omega)$ be a rearrangement-invariant quasi-normed
function space. Assume that $X(\Omega)$ satisfies the local
integrability axiom: for every measurable set
$E\subset\Omega$ with $|E|<\infty$, there exists a constant
$C_E>0$ such that
\begin{equation*}
    \int_E |f(x)|\,\mathrm{d}x
    \le C_E\| f\|_{X(\Omega)}
    \qquad\text{for every }f\in X(\Omega).
\end{equation*}
Then
\begin{equation*}
    X(\Omega)\hookrightarrow L^1(\Omega).
\end{equation*}
More precisely, there exists a constant $C_X>0$, depending only on
$X(\Omega)$ and $\Omega$, such that
\begin{equation*}
    \| f\|_{L^1(\Omega)}
    \le C_X\| f\|_{X(\Omega)}
    \qquad\text{for every }f\in X(\Omega).
\end{equation*}
\end{lemma}

\begin{proof}
Let $v\in X(\Omega)$. For every $k\in\mathbb N$, define the
nonnegative truncation
\[
    v_k(x)=\min\{|v(x)|,k\},
    \qquad x\in\Omega.
\]
Since
\[
    0\le v_k(x)\le |v(x)|
    \qquad\text{for almost every }x\in\Omega,
\]
the lattice property of $X(\Omega)$ implies that
\[
    v_k\in X(\Omega)
    \qquad\text{and}\qquad
    \lVert v_k\rVert_{X(\Omega)}
    \le \lVert v\rVert_{X(\Omega)}.
\]

Since $|\Omega|<\infty$, we can apply with $E=\Omega$. Hence, for every $k\in\mathbb N$,
\begin{align}
    \int_\Omega v_k(x)\,\mathrm{d}x
    &\le C_\Omega\lVert v_k\rVert_{X(\Omega)}
    \notag\\
    &\le C_\Omega\lVert v\rVert_{X(\Omega)}.
    \label{app:eq:truncation-L1-bound}
\end{align}
The constant $C_\Omega$ is independent of both $v$ and $k$.

Moreover,
\[
    0\le v_k(x)\rightarrow |v(x)|
    \qquad\text{for almost every }x\in\Omega
\]
as $k\to\infty$. Therefore, the monotone convergence theorem and
\eqref{app:eq:truncation-L1-bound} yield
\begin{align}
    \lVert v\rVert_{L^1(\Omega)}
    &=\int_\Omega |v(x)|\,\mathrm{d}x
    \notag\\
    &=\lim_{k\to\infty}
      \int_\Omega v_k(x)\,\mathrm{d}x
    \notag\\
    &\le C_\Omega\lVert v\rVert_{X(\Omega)}.
    \label{app:eq:X-L1-detailed}
\end{align}
In particular, the integral on the left-hand side is finite.
Thus, every function in $X(\Omega)$ belongs to $L^1(\Omega)$, and
the inclusion map
\begin{equation*}
    X(\Omega)\hookrightarrow L^1(\Omega).
\end{equation*}
is continuous. This completes the proof.

\end{proof}

\begin{corollary}\label{app:prop:ri-reduction}
Let $X(\Omega)$ and $Y(\Omega)$ be
rearrangement-invariant quasi-normed spaces, with representation spaces
$\overline X(0,|\Omega|)$ and $\overline Y(0,|\Omega|)$. Assume that, for every nonnegative function $f \in X(\Omega)$, if $2 \leq p<n$, and 
\begin{equation}\label{eq:Af(s)1}
\left\|s^{-\frac{1}{n'}}\int_0^s f^*(r)\,\dd r\right\|
 _{\overline Y(0,|\Omega|)} \leq
 C\|f^*\|_{\overline X(0,|\Omega|)}
\end{equation}
with
\begin{equation}\label{eq:Bf(s)1}
\left\|\int_s^m f^*(r)r^{-\frac{1}{n'}}\,\dd r\right\| _{\overline Y(0,|\Omega|)} \leq C\|f^*\|_{\overline X(0,|\Omega|)}
\end{equation}
or for $1<p<2$, and
\begin{equation}\label{eq:Af(s)2}
 \left\|
 s^{-\frac{1}{n'}}
 \left(\int_0^s f^*(r)\,\mathrm{d}r\right)^{p-1}
 \right\|_{\overline Y(0,M)}
 \le
 C\lVert f^*\rVert_{\overline X(0,M)}^{p-1}
\end{equation}
with
\begin{equation}\label{eq:Bf(s)2}
\left\|\int_s^m f^*(r)r^{-\frac{1}{n'}}\,\dd r\right\| _{\overline Y(0,|\Omega|)} \leq C\|f^*\|_{\overline X(0,|\Omega|)}.
\end{equation}
Then
\begin{equation} \label{app:eq:ri-conclusion}
 \| H(\nabla u)^{p-1}\|_{Y(\Omega)}
 \le C\|f\|_{X(\Omega)}.
\end{equation}
\end{corollary}

\begin{proof}
With the notation before, one has
\begin{equation*}
 A_f(s)=\int_0^s f^*(r)\,\mathrm{d}r,
 \qquad
 B_f(s)=\int_s^m f^*(r)r^{-\frac{1}{n'}}\,\mathrm{d}r.
\end{equation*}
Consequently, when $2\le p<n$, the first line of Theorem \ref{thm:main}
becomes directly
\[
 \big([H(\nabla u)]^*(s)\big)^{p-1}
 \le C\bigl(s^{-\frac{1}{n'}}A_f(s)+B_f(s)\bigr).
\]
Next, by the defining property of the representation space
$\overline Y(0,m)$,
\begin{equation*}
 \| H(\nabla u)^{p-1}\|_{Y(\Omega)}
 =\| (|H(\nabla u)|^*)^{p-1}\|_{\overline Y(0,m)}.
\end{equation*}
By \eqref{eq:Af(s)1}, \eqref{eq:Bf(s)1}, lattice monotonicity of the representation quasi-norm, and its quasi-triangle inequality,
\begin{align}
\begin{split}
 \| \big(H(\nabla u)(s)\big)^{p-1}\|_{Y(\Omega)}
 &=\|\big([H(\nabla u)]^*(s)\big)^{p-1}\|_{\overline Y(0,m)}\\
 &\leq C\left(
 \| s^{-\frac{1}{n'}}A_f(s)\|_{\overline Y(0,m)}
 +\| B_f(s)\|_{\overline Y(0,m)}
 \right)\\
 &\leq  C\|f^*\|_{\overline X(0,|\Omega|)}\\
 &= C\|f\|_{X(\Omega)}.
 \end{split}
\end{align}
When $1<p<2$, the second line of Theorem \ref{thm:main} reads
\begin{equation*}
\big([H(\nabla u)]^*(s)\big)^{p-1}
 \le C\left(s^{-\frac{1}{n'}}\| f\|^{2-p}_{L^1(\Omega)}A_f(s)^{p-1}
   +B_f(s)\right)
\end{equation*}
Using the lattice monotonicity of $\overline Y(0,M)$, and its quasi-triangle
inequality, we obtain
\begin{align}
\begin{split}
 \bigl\|H(\nabla u)^{p-1}\bigr\|_{Y(\Omega)}
 &=
 \left\|
 \bigl(H(\nabla u)^{p-1}\bigr)^*
 \right\|_{\overline Y(0,M)}
 \notag\\
 &\le
 C\| f\rVert_{L^1(\Omega)}^{2-p}
 \left\|
 s^{-\frac{1}{n'}}
 A_f(s)^{p-1}
 \right\|_{\overline Y(0,M)}\\
 &\quad+
 C\left\|
 B_f(s)
 \right\|_{\overline Y(0,M)}.
\end{split}
\end{align}
Applying \eqref{eq:Af(s)2}, \eqref{eq:Bf(s)2} and Lemma \ref{app:lem:X-into-L1} yields
\begin{align}\label{app:eq:singular-before-final-embedding}
\begin{split}
 \bigl\|H(\nabla u)^{p-1}\bigr\|_{Y(\Omega)}
 &\le
 C\| f\|_{L^1(\Omega)}^{2-p}
 \| f\|_{X(\Omega)}^{p-1}
 +
 C\| f\|_{X(\Omega)}\\
 &\leq C
 \| f\|_{X(\Omega)}^{2-p}
 \| f\|_{X(\Omega)}^{p-1} + C\| f\rVert_{X(\Omega)}
 \\
 &=
 C\| f\|_{X(\Omega)}.
 \end{split}
\end{align}
Hence, the proof is complete.
\end{proof}

\begin{table}[H]
\centering
\caption{Comparison of degenerate and
singular ranges.}
\label{tab:comparison-degenerate-singular}

\begingroup
\small
\setlength{\tabcolsep}{5pt}
\renewcommand{\arraystretch}{1.25}

\def\cellmath#1{%
    \par\smallskip
    {\centering
    \(\displaystyle #1\)\par}%
    \smallskip
}

\begin{tabular}{
    @{}
    >{\raggedright\arraybackslash
      \hyphenpenalty=10000\relax
      \exhyphenpenalty=10000\relax}
      p{0.18\linewidth}
    >{\raggedright\arraybackslash
      \hyphenpenalty=10000\relax
      \exhyphenpenalty=10000\relax}
      p{0.38\linewidth}
    >{\raggedright\arraybackslash
      \hyphenpenalty=10000\relax
      \exhyphenpenalty=10000\relax}
      p{0.38\linewidth}
    @{}
}
\toprule
\textbf{Assumption on the datum}
&
\textbf{Degenerate range \(2\leq p<n\)}
&
\textbf{Singular range \(1<p<2\)}
\\
\midrule
\addlinespace[1ex]

\(q=1\)
&
Weak endpoint:
\cellmath{
    G\in L^{n'(p-1),\infty}(\Omega).
}
Strong subcritical estimates:
\cellmath{
    \begin{gathered}
        G\in L^\sigma(\Omega),\\[-0.3ex]
        \text{for every }0<\sigma<n'(p-1).
    \end{gathered}
}
&
The same 
\\

\addlinespace[4ex]

\(1<q<n\)
&
Strong endpoint estimate:
\cellmath{
    G\in L^{ \frac{nq(p-1)}{n-q}}(\Omega)
}

&
Weak endpoint estimate:
\cellmath{
    G\in
    L^{\sigma_{p,q},\infty}(\Omega).
}
Strong subcritical estimates:
\cellmath{
    \begin{gathered}
        G\in L^\sigma(\Omega),\\[-0.3ex]
        \text{for every }
        0<\sigma<\sigma_{p,q} = \frac{p-1}{
        \frac{1}{n'}
        -
        \frac{p-1}{q'}}.
    \end{gathered}
}
\\

\addlinespace[4ex]

\(q=n\)
&
Strong endpoint estimate:
\cellmath{
    \begin{gathered}
        G\in L^\sigma(\Omega),\\[-0.3ex]
        \text{for every }0<\sigma<\infty.
    \end{gathered}
}

&
Weak endpoint estimate:
\cellmath{
    G\in
    L^{
        \frac{n'(p-1)}{2-p},
        \infty
    }(\Omega).
}
Strong subcritical estimates:
\cellmath{
    \begin{gathered}
        G\in L^\sigma(\Omega),\\[-0.3ex]
        \text{for every }
        0<\sigma<
        \frac{n'(p-1)}{2-p}.
    \end{gathered}
}
\\

\addlinespace[4ex]

\(q>n\)
&
The rearrangement estimate directly yields
\cellmath{
    G\in L^\infty(\Omega).
}
&
Can't get \(L^\infty\) estimate in general,
Unless 
\cellmath{
    \begin{gathered}
         G\in L^\infty(\Omega),\\[-0.3ex]
        \text{for every }
        \frac{p-1}{q'} \geq \frac{1}{n'}.
    \end{gathered}
}
\\

\addlinespace[0.8ex]
\bottomrule
\end{tabular}

\endgroup
\end{table}

\begin{theorem}\label{thm:anisotropic-lebesgue-estimates}
Let $\Omega$, $H$, and $u$ be as in Theorem~1.1, and assume that $2\leq p<n$. If $f\in L^q(\Omega)$, then the following assertions hold:
\begin{enumerate}
    \item[\textup{(i)}]
    If $q=1$, then
    \[
    \|H(\nabla u)\|_{L^{n'(p-1),\infty}(\Omega)}
    \leq
    C\|f\|_{L^1(\Omega)}^{\frac{1}{p-1}}.
\]
    moreover, for every
    \[
        0<\sigma<\frac{n(p-1)}{n-1},
    \]
    one has
    \[
        \|H(\nabla u)\|_{L^\sigma(\Omega)}
        \leq
        C\|f\|_{L^1(\Omega)}^{\frac{1}{p-1}}.
    \]

    \item[\textup{(ii)}]
    If $1<q<n$, then
    \[
        \|H(\nabla u)\|_
        {L^{\frac{nq(p-1)}{n-q}}(\Omega)}
        \leq
        C\|f\|_{L^q(\Omega)}^{\frac{1}{p-1}}.
    \]

    \item[\textup{(iii)}]
    If $q=n$, then, for every $0<\sigma<\infty$,
    \[
        \|H(\nabla u)\|_{L^\sigma(\Omega)}
        \leq
        C_\sigma\|f\|_{L^n(\Omega)}^{\frac{1}{p-1}}.
    \]

    \item[\textup{(iv)}]
    If $q>n$, then
    \[
        \|H(\nabla u)\|_{L^\infty(\Omega)}
        \leq
        C\|f\|_{L^q(\Omega)}^{\frac{1}{p-1}}.
    \]
\end{enumerate}
The constants are independent of $f$ and $u$. They may depend on the
structural and geometric parameters occurring in Theorem~1.1 and on the
displayed exponents.
\end{theorem}

\begin{proof}
Set $m =|\Omega|$ and define
\[
    A_f(s)=\int_0^s f^*(r)\,dr,
    \qquad
    B_f(s)=\int_s^m f^*(r)r^{-\alpha}\,dr,
    \qquad 0<s<m.
\]
By the first line of the rearrangement estimate in Theorem~1.1,
\begin{equation}
\label{eq:anisotropic-basic-rearrangement}
    \big([H(\nabla u)]^*(s)\big)^{p-1}
    \leq
    C\left[
        s^{-\frac{1}{n'}}A_f(s)+B_f(s)
    \right],
    \qquad 0<s<m.
\end{equation}

\medskip
\noindent
\emph{Case \textup{(i)}: $q=1$}

Since $A_f(s)$ is only the integral of $f^*$ over $(0,s)$, whereas
$\|f\|_{L^1(\Omega)}$ is the integral over the whole interval $(0,m)$,
we have
\begin{equation}
\label{eq:Af-L1-bound}
    A_f(s)
    =
    \int_0^s f^*(r)\,dr
    \leq
    \int_0^m f^*(r)\,dr
    =
    \|f\|_{L^1(\Omega)}.
\end{equation}
Since $f^*$ is nonincreasing, for every $r\in(0,m)$, we have
\begin{align}
    f^*(r)
    \leq
    \frac{1}{r}\int_0^r f^*(t)\,dt
    \leq
    \frac{\|f\|_{L^1(\Omega)}}{r}.
\end{align}
It follows that
\begin{align*}
    B_f(s)
    &=
    \int_s^m f^*(r)r^{-\frac{1}{n'}}\,dr\\
    &\leq
    \|f\|_{L^1(\Omega)}
    \int_s^m r^{-1-\frac{1}{n'}}\,dr\\
    &=
    \frac{\|f\|_{L^1(\Omega)}}{\frac{1}{n'}}
    \left(s^{-\frac{1}{n'}}-m^{-\frac{1}{n'}}\right)\\
    &\leq
    C \|f\|_{L^1(\Omega)}s^{-\frac{1}{n'}}.
\end{align*}
Combining this estimate with \eqref{eq:Af-L1-bound} and
\eqref{eq:anisotropic-basic-rearrangement}, we obtain
\begin{equation}
\label{eq:l-one-pointwise-anisotropic}
    \big([H(\nabla u)]^*(s)\big)^{p-1}
    \leq
    C\|f\|_{L^1(\Omega)}s^{-\frac{1}{n'}},
    \qquad 0<s<m.
\end{equation}
Since the function $t\mapsto t^{p-1}$ is strictly
increasing on $[0,\infty)$, decreasing rearrangement commutes with
this power. More precisely,
\begin{equation}
\label{eq:rearrangement-power-identity}
    \big([H(\nabla u)]^{p-1}\big)^*(s)
    =
    \big([H(\nabla u)]^*(s)\big)^{p-1},
    \qquad 0<s<m.
\end{equation}
In particular, recalling the definition of Lorentz norm and \eqref{eq:l-one-pointwise-anisotropic}, we get
\begin{align}\label{eq:H-inequality}
    \|[H(\nabla u)]^{p-1}\|_{L^{n',\infty}(\Omega)}
    =
    \sup_{0<s<m}s^{\frac{1}{n'}}\big([H(\nabla u)]^{p-1}\big)^*(s)
    \leq
    C\|f\|_{L^1(\Omega)}.
\end{align}

Moreover,
\begin{align*}
    \|[H(\nabla u)]^{p-1}\|_{L^{n',\infty}(\Omega)}
    &=
    \sup_{0<s<m}
    s^{\frac{1}{n'}}
    \big([H(\nabla u)]^{p-1}\big)^*(s)\\
    &=
    \sup_{0<s<m}
    s^{\frac{1}{n'}}
   \big([H(\nabla u)]^*(s)\big)^{p-1}\\
    &=
    \sup_{0<s<m}
    \left(
        s^{1/[n'(p-1)]}[H(\nabla u)]^*(s)
    \right)^{p-1}\\
    &=
    \|H(\nabla u)\|_{L^{n'(p-1),\infty}(\Omega)}^{p-1}.
\end{align*}
Consequently, by \eqref{eq:H-inequality} we have
\[
    \|H(\nabla u)\|_{L^{n'(p-1),\infty}(\Omega)}^{p-1}
    \leq
    C\|f\|_{L^1(\Omega)}.
\]
Taking the power $\frac{1}{p-1}$ on both sides and renaming the resulting
constant, we obtain
\[
    \|H(\nabla u)\|_{L^{n'(p-1),\infty}(\Omega)}
    \leq
    C\|f\|_{L^1(\Omega)}^{\frac{1}{p-1}}.
\]

Let now
\[
    0<\sigma<n'(p-1)
    =
    \frac{n(p-1)}{n-1}
\]
and set
\[
    \rho=\frac{\sigma}{p-1}.
\]
Then $0<\rho<n'$. By \eqref{eq:l-one-pointwise-anisotropic},
\begin{align*}
\begin{split}
    \|[H(\nabla u)]^{p-1}\|_{L^\rho(\Omega)}^\rho
    &=
    \int_0^m \big([H(\nabla u)]^{p-1}\big)^*(s)\bigr)^\rho\,ds\\
    &\leq
    C\|f\|_{L^1(\Omega)}^\rho
    \int_0^m s^{-\rho/n'}\,ds\\
    &=
    \frac{Cm^{1-\rho/n'}}{1-\rho/n'}
    \|f\|_{L^1(\Omega)}^\rho.
    \end{split}
\end{align*}
Therefore,
\[
    \|[H(\nabla u)]^{p-1}\|_{L^\rho(\Omega)}
    \leq
    C\|f\|_{L^1(\Omega)}.
\]
Since $\rho(p-1)=\sigma$, 
\begin{align*}
    \|[H(\nabla u)]^{p-1}\|_{L^\rho(\Omega)}
    &=
    \left(
        \int_\Omega
        [H(\nabla u)]^{\rho(p-1)}\,dx
    \right)^{1/\rho}\\
    &=
    \left(
        \int_\Omega
        [H(\nabla u)]^\sigma\,dx
    \right)^{1/\rho}\\
    &=
    \left(
        \int_\Omega
        [H(\nabla u)]^\sigma\,dx
    \right)^{(p-1)/\sigma}\\
    &=
    \|H(\nabla u)\|_{L^\sigma(\Omega)}^{p-1}.
\end{align*}
Consequently,
\[
    \|H(\nabla u)\|_{L^\sigma(\Omega)}^{p-1}
    \leq
    C\|f\|_{L^1(\Omega)}.
\]
Taking the power $\frac{1}{p-1}$ on both sides and renaming the resulting
constant, we conclude that
\[
    \|H(\nabla u)\|_{L^\sigma(\Omega)}
    \leq
    C\|f\|_{L^1(\Omega)}^{\frac{1}{p-1}}.
\]

\medskip
\noindent
\emph{Case \textup{(ii)}: $1<q<n$}

By H\"older's inequality,
\begin{align}\label{eq:first-hardy-critical-decay}
\begin{split}
    A_f(s)
    &=
    s^{-\frac{1}{n'}}\int_0^sf^*(r)\,dr
    \notag\\
    &\leq
    s^{-\frac{1}{n'}}
    \left(
        \int_0^s|f^*(r)|^q\,dr
    \right)^{\frac1q}
    s^{\frac{1}{q'}}
    \notag\\
    &\leq
    s^{\frac{1}{q'}-\frac{1}{n'}}
    \|f\|_{L^q(0,M)}.
\end{split}
\end{align}
The exponent of $s$ satisfies
\begin{align}
    \frac{1}{q'}-\frac{1}{n'}
    &=
    \left(1-\frac{1}{q}\right)
    -
    \left(1-\frac{1}{n}\right)
    \notag\\
    &=
    -\left(
        \frac{1}{q}-\frac{1}{n}
    \right).
\end{align}
Since $q<n$, the number
\[
    \theta_q
    :=
    \frac{1}{q}-\frac{1}{n}>0.
\]
 Hence
\eqref{eq:first-hardy-critical-decay} takes the form
\[
    A_f(s)
    \leq
    s^{-\theta_q}
    \|f\|_{L^q(0,m)}.
\]

The same critical power occurs for the second Hardy operator. Indeed,
since $q<n$, one has $\frac{q'}{n'}>1$, and H\"older's inequality gives
\begin{align}\label{eq:second-hardy-critical-decay}
\begin{split}
    B_f(s)
    &=
    \int_s^m f(r)r^{-\frac{1}{n'}}\,dr\\
    &\leq
    \|f\|_{L^q(0,m)}
    \left(
        \int_s^m r^{-\frac{q'}{n'}}\,dr
    \right)^{\frac{1}{q'}}\\
    &\leq
    C\|f\|_{L^q(0,m)}
    s^{\frac{1}{q'}-\frac{1}{n'}}\\
    &=
    C\|f\|_{L^q(0,m)}
    s^{-\theta_q}.
\end{split}
\end{align}

Thus both Hardy operators have the critical decay
$s^{-\theta_q}$. Since the critical power associated with weak
$L^r(0,m)$ is $s^{-\frac1r}$, the natural target exponent $r$ is
determined by
\[
    \frac{1}{r}
    =
    \theta_q
    =
    \frac{1}{q}-\frac{1}{n}.
\]
Solving this identity gives
\[
    r=\frac{nq}{n-q}.
\]
From now on, we denote this naturally determined exponent by
\begin{equation}
\label{eq:q-star-naturally-derived}
    q^\star
    :=
    \frac{nq}{n-q},
    \qquad
    \frac{1}{q^\star}
    =
    \frac{1}{q}-\frac{1}{n}.
\end{equation}
The strong $L^{q^\star}$ estimates follow from the classical weighted
Hardy inequalities:
\begin{equation}
\label{eq:two-hardy-strong-estimates}
    \|A_f\|_{L^{q^\star}(0,M)}
    +
    \|B_f\|_{L^{q^\star}(0,M)}
    \leq
    C\|f\|_{L^q(0,M)}.
\end{equation}
For completeness, the criteria for the two inequalities in
\eqref{eq:two-hardy-strong-estimates} are
\begin{equation}
\label{eq:first-hardy-criterion-natural}
    \sup_{0<t<M}
    \left(
        \int_t^M s^{-\frac{q^\star}{n'}}\,ds
    \right)^{\frac{1}{q^\star}}
    t^{\frac{1}{q'}}
    <\infty
\end{equation}
and
\begin{equation}
\label{eq:second-hardy-criterion-natural}
    \sup_{0<t<M}
    t^{\frac{1}{q^\star}}
    \left(
        \int_t^M r^{-\frac{q'}{n'}}\,dr
    \right)^{\frac{1}{q'}}
    <\infty.
\end{equation}
The first condition is satisfied because $q>1$ implies
\[
    \frac{q^\star}{n'}>1,
\]
whereas the second is satisfied because $q<n$ implies
\[
    \frac{q'}{n'}>1.
\]
Moreover, in both cases the remaining power of $t$ is zero, since
\begin{equation}
\label{eq:hardy-balance-natural}
    \frac{1}{q^\star}
    +
    \frac{1}{q'}
    -
    \frac{1}{n'}
    =0.
\end{equation}
This proves \eqref{eq:two-hardy-strong-estimates}.

Hence we obtain that
\begin{align}\label{eq:w-q-star-natural}
\begin{split}
    \|\big([H(\nabla u)]^{p-1}\big)\|_{L^{q^\star}(\Omega)}
    &=
    \|\big([H(\nabla u)]^{p-1}\big)^*(s)\|_{L^{q^\star}(0,M)}\\
    &\leq
    C\left(
        \|A_f\|_{L^{q^\star}(0,M)}
        +
        \|B_f\|_{L^{q^\star}(0,M)}
    \right)\\
    &\leq
    C\|f\|_{L^q(\Omega)}.
    \end{split}
\end{align}

Using \eqref{eq:rearrangement-power-identity}, a direct calculation
gives
\begin{align}
\begin{split}
    \|[H(\nabla u)]^{p-1}\|_{L^{q^\star}(\Omega)}
    &=
    \left(
        \int_\Omega
        [H(\nabla u)]^{q^\star(p-1)}\,dx
    \right)^{\frac{1}{q^\star}}\\
    &=
    \|H(\nabla u)\|_
    {L^{q^\star(p-1)}(\Omega)}^{p-1}.
\end{split}
\end{align}
Combining \eqref{eq:w-q-star-natural}, we conclude that
\begin{equation}
\label{eq:intermediate-final-natural}
    \|H(\nabla u)\|_
    {L^{q^\star(p-1)}(\Omega)}
    \leq
    C\|f\|_{L^q(\Omega)}^{\frac{1}{p-1}}.
\end{equation}
Finally, by \eqref{eq:q-star-naturally-derived},
\[
    q^\star(p-1)
    =
    \frac{nq(p-1)}{n-q}.
\]
Therefore,
\[
    \|H(\nabla u)\|_
    {L^{\frac{nq(p-1)}{n-q}}(\Omega)}
    \leq
    C\|f\|_{L^q(\Omega)}^{\frac{1}{p-1}}.
\]
This proves assertion \textup{(ii)}.

\medskip
\noindent
\emph{Case \textup{(iii)}: $q=n$}

Fix $0<\sigma<\infty$. Choose
$q_0\in(1,n)$ sufficiently close to $n$ so that
\[
    \tau
    :=
    \frac{nq_0(p-1)}{n-q_0}
    >
    \sigma.
\]
Since $|\Omega|=M<\infty$,
\[
    \|f\|_{L^{q_0}(\Omega)}
    \leq
    M^{\frac{1}{q_0}-\frac1n}
    \|f\|_{L^n(\Omega)}.
\]
Applying assertion \textup{(ii)} with $q=q_0$, we obtain
\[
    \|H(\nabla u)\|_{L^\tau(\Omega)}
    \leq
    C\|f\|_{L^n(\Omega)}^{\frac{1}{p-1}}.
\]
Since $\tau>\sigma$ and $\Omega$ has finite measure,
\[
    \|H(\nabla u)\|_{L^\sigma(\Omega)}
    \leq
    M^{\frac{1}{\sigma}-\frac{1}{\tau}}
    \|H(\nabla u)\|_{L^\tau(\Omega)}.
\]
It follows that
\[
    \|H(\nabla u)\|_{L^\sigma(\Omega)}
    \leq
    C_\sigma
    \|f\|_{L^n(\Omega)}^{\frac{1}{p-1}}.
\]
This proves assertion \textup{(iii)}.

\medskip
\noindent
\emph{Case \textup{(iv)}: $q>n$}

 By H\"older's inequality, we get the same result of first decay as \eqref{eq:first-hardy-critical-decay},
\begin{align}
\begin{split}
    A_f(s) &\leq 
    s^{\frac1n-\frac1q}
    \|f\|_{L^q(\Omega)}\\
    &\leq  M^{\frac1n-\frac1q}
    \|f\|_{L^q(\Omega)}
\end{split}
\end{align}
Moreover, we still get the same result of second decay as \eqref{eq:second-hardy-critical-decay}
\begin{align*}
    B_f(s)
    &=
    \int_s^M f^*(r)r^{-\frac{1}{n'}}\,dr\\
    &\leq
    \|f^*\|_{L^q(0,M)}
    \left(
        \int_0^M r^{-\frac{q'}{n'}}\,dr
    \right)^{1/q'}\\
    &\leq C\|f\|_{L^q(0,M)},
\end{align*}
where $q>n$, and $\frac{q'}{n'}<1$, one has the last integral is finite. Hence, we conclude that
\[
    \|H(\nabla u)\|_{L^\infty(\Omega)}
    \leq
    C\|f\|_{L^q(\Omega)}^{\frac{1}{p-1}}.
\]
This proves assertion \textup{(iv)} and completes the proof.
\end{proof}

\begin{theorem}
\label{thm:singular-gradient-estimates}

Let \(\Omega\), \(H\), \(f\), and \(u\) be as in Theorem~1.1, and
assume that $1<p<2$.
If \(f\in L^q(\Omega)\), then the following assertions hold.

\begin{enumerate}
\renewcommand{\labelenumi}{\textnormal{(\roman{enumi})}}

\item
If \(q=1\), then
\[
    \left\|H(\nabla u)\right\|_{
        L^{n'(p-1),\infty}(\Omega)
    }
    \leq
    C
    \|f\|_{L^1(\Omega)}^{\frac{1}{p-1}}.
\]
Moreover, for every
\[
    0<\sigma<n'(p-1),
\]
one has
\[
    \left\|H(\nabla u)\right\|_{L^\sigma(\Omega)}
    \leq
    C_\sigma
    \|f\|_{L^1(\Omega)}^{\frac{1}{p-1}}.
\]

\item
If \(1<q<n\), let
\[
    q'=\frac{q}{q-1}
\]
and define
\[
    \delta_{p,q}
    :=
    \frac{1}{n'}
    -
    \frac{p-1}{q'}.
\]
Since \(1<p<2\) and \(1<q<n\), one has
\[
    \delta_{p,q}>0.
\]
Set
\[
    \sigma_{p,q}
    :=
    \frac{p-1}{\delta_{p,q}}
    =
    \frac{p-1}{
        \displaystyle
        \frac{1}{n'}
        -
        \frac{p-1}{q'}
    }.
\]
Then
\[
    \left\|H(\nabla u)\right\|_{
        L^{\sigma_{p,q},\infty}(\Omega)
    }
    \leq
    C
    \|f\|_{L^q(\Omega)}^{\frac{1}{p-1}}.
\]
Consequently, for every
\[
    0<\sigma<\sigma_{p,q},
\]
one has
\[
    \left\|H(\nabla u)\right\|_{L^\sigma(\Omega)}
    \leq
    C_\sigma
    \|f\|_{L^q(\Omega)}^{\frac{1}{p-1}}.
\]

\item
If \(q=n\), define
\[
    \sigma_{p,n}
    :=
    \frac{n'(p-1)}{2-p}.
\]
Then
\begin{equation}
\label{eq:singular-weak-intermediate}
    \|H(\nabla u)\|_
    {L^{\sigma_{p,q},\infty}(\Omega)}
    \leq
    C\|f\|_{L^q(\Omega)}^{\frac{1}{p-1}}.
\end{equation}
Consequently, for every
\[
    0<\sigma<\sigma_{p,n}
    =
    \frac{n'(p-1)}{2-p},
\]
one has
\begin{equation}
\label{eq:singular-strong-intermediate}
    \|H(\nabla u)\|_{L^\sigma(\Omega)}
    \leq
    C_\sigma
    \|f\|_{L^q(\Omega)}^{\frac{1}{p-1}}.
\end{equation}

\item
If \(q>n\), and one additionally invokes
the global Lorentz endpoint estimate
\[
    \left\|H(\nabla u)\right\|_{L^\infty(\Omega)}
    \leq
    C
    \|f\|_{L^{n,1}(\Omega)}^{\frac{1}{p-1}},
\]
then the embedding
\[
    L^q(\Omega)\hookrightarrow L^{n,1}(\Omega),
    \qquad q>n,
\]
gives
\[
    \left\|H(\nabla u)\right\|_{L^\infty(\Omega)}
    \leq
    C
    \|f\|_{L^q(\Omega)}^{\frac{1}{p-1}}.
\]

\end{enumerate}

The constants are independent of \(f\) and \(u\). They may depend on
\(n\), \(p\), \(q\), \(\sigma\), \(|\Omega|\), and on the structural
and geometric parameters occurring in Theorem~1.1.
\end{theorem}

\begin{proof}
Set
\[
    M=|\Omega|,
\]
Define
\[
    A_f(s)
    =
    \int_0^s f^*(r)\,dr
\]
and
\[
    B_f(s)
    =
    \int_s^M f^*(r)r^{-\frac{1}{n'}}\,dr.
\]

\medskip
\noindent
\emph{Case \textup{(i)}: $q=1$}

The only difference from the proof in the degenerate range is the
first term
\[
    s^{-\frac{1}{n'}}\|f\|_{L^1(\Omega)}^{2-p}A_f(s)^{p-1}.
\]
Since $p-1>0$, it follows from
\eqref{eq:Af-L1-bound} that
\[
    A_f(s)^{p-1}\leq \|f\|_{L^1(\Omega)}^{p-1}.
\]
Consequently,
\begin{align}
\begin{split}
    s^{-\frac{1}{n'}}\|f\|_{L^1(\Omega)}^{2-p}A_f(s)^{p-1}
    &\leq
    s^{-\frac{1}{n'}}\|f\|_{L^1(\Omega)}^{2-p}\|f\|_{L^1(\Omega)}^{p-1}\\
    &=
  s^{-\frac{1}{n'}} \|f\|_{L^1(\Omega)}.
\end{split}
\end{align}

The second term $B_f$ is identical to the one appearing in the degenerate range.  Hence we omit its repetition. This proves the assertion for $q=1$ in
the singular range has the same result as degenerate range.

\medskip
\noindent
\emph{Case \textup{(ii)}: $1<q<n$}

By the singular line of the rearrangement estimate in Theorem~1.1,
\begin{equation}
\label{eq:singular-main-intermediate}
    \big([H(\nabla u)]^*(s)\big)^{p-1}
    \leq
    C\left[
        s^{-\frac{1}{n'}}
        \|f\|_{L^1(\Omega)}^{2-p}
        A_f(s)^{p-1}
        +
        B_f(s)
    \right],
    \qquad 0<s<M.
\end{equation}

We estimate the two terms on the right-hand side separately. For the first term of the right side, by H\"older's inequality,
\begin{align}
    A_f(s)
    &=
    \int_0^s f^*(r)\,dr
    \notag\\
    &\leq
    \left(
        \int_0^s [f^*(r)]^q\,dr
    \right)^{\frac1q}
    \left(
        \int_0^s1^{q'}\,dr
    \right)^{\frac{1}{q'}}
    \notag\\
    &\leq
    s^{\frac{1}{q'}}
    \|f^*\|_{L^q(0,M)}
    \notag\\
    &=
    s^{\frac{1}{q'}}
    \|f\|_{L^q(\Omega)}.
\label{eq:singular-Af-q-intermediate}
\end{align}
Moreover, since $|\Omega|=M<\infty$, the finite-measure embedding
$L^q(\Omega)\hookrightarrow L^1(\Omega)$ gives
\begin{equation}
\label{eq:singular-Lq-L1-intermediate}
    \|f\|_{L^1(\Omega)}
    \leq
    M^{\frac{1}{q'}}
    \|f\|_{L^q(\Omega)}.
\end{equation}
Combining
\eqref{eq:singular-Af-q-intermediate} and
\eqref{eq:singular-Lq-L1-intermediate}, we obtain
\begin{align}
    &s^{-\frac{1}{n'}}
    \|f\|_{L^1(\Omega)}^{2-p}
    A_f(s)^{p-1}
    \notag\\
    &\quad\leq
    s^{-\frac{1}{n'}}
    \left(
        M^{\frac{1}{q'}}\|f\|_{L^q(\Omega)}
    \right)^{2-p}
    \left(
        s^{\frac{1}{q'}}\|f\|_{L^q(\Omega)}
    \right)^{p-1}
    \notag\\
    &\quad=
    M^{\frac{2-p}{q'}}
    s^{-\frac{1}{n'}+\frac{p-1}{q'}}
    \|f\|_{L^q(\Omega)}^{(2-p)+(p-1)}
    \notag\\
    &\quad=
    M^{\frac{2-p}{q'}}
    \|f\|_{L^q(\Omega)}
    s^{-\delta_{p,q}},
\label{eq:singular-head-q-intermediate}
\end{align}
and the definition
\[
    \delta_{p,q}
    =
    \frac{1}{n'}
    -
    \frac{p-1}{q'}.
\]
Thus
\begin{equation}
\label{eq:singular-head-final-intermediate}
    s^{-\frac{1}{n'}}
    \|f\|_{L^1(\Omega)}^{2-p}
    A_f(s)^{p-1}
    \leq
    C\|f\|_{L^q(\Omega)}
    s^{-\delta_{p,q}}.
\end{equation}

For the second term of right side, again by H\"older's inequality,
\begin{align}
    B_f(s)
    &=
    \int_s^M f^*(r)r^{-\frac{1}{n'}}\,dr
    \notag\\
    &\leq
    \left(
        \int_s^M [f^*(r)]^q\,dr
    \right)^{\frac1q}
    \left(
        \int_s^M r^{-\frac{q'}{n'}}\,dr
    \right)^{\frac{1}{q'}}
    \notag\\
    &\leq
    \|f\|_{L^q(\Omega)}
    \left(
        \int_s^M r^{-\frac{q'}{n'}}\,dr
    \right)^{\frac{1}{q'}}.
\label{eq:singular-tail-holder-intermediate}
\end{align}
Since $q<n$, one has
\[
    \frac{q'}{n'}>1.
\]
Therefore,
\begin{align*}
    \int_s^M r^{-\frac{q'}{n'}}\,dr
    &=
    \frac{
        s^{1-\frac{q'}{n'}}-M^{1-\frac{q'}{n'}}
    }{
        \frac{q'}{n'}-1
    }\\
    &\leq
    C s^{1-\frac{q'}{n'}}.
\end{align*}
Substituting this estimate into
\eqref{eq:singular-tail-holder-intermediate}, we obtain
\begin{align}
    B_f(s)
    &\leq
    C\|f\|_{L^q(\Omega)}
    s^{\frac{1}{q'}-\frac{1}{n'}}
    \notag\\
    &=
    C\|f\|_{L^q(\Omega)}
    s^{-\eta_q},
\label{eq:singular-tail-eta-intermediate}
\end{align}
where
\begin{equation}
\label{eq:eta-q-definition}
    \eta_q
    :=
    \frac{1}{n'}
    -
    \frac{1}{q'}
    =
    \frac{1}{q}
    -
    \frac{1}{n}.
\end{equation}
We now compare $\eta_q$ with $\delta_{p,q}$. A direct calculation gives
\begin{align}
    \delta_{p,q}-\eta_q
    &=
    \left(
        \frac{1}{n'}
        -
        \frac{p-1}{q'}
    \right)
    -
    \left(
        \frac{1}{n'}
        -
        \frac{1}{q'}
    \right)
    \notag\\
    &=
    \frac{2-p}{q'}
    \notag\\
    &=
    (2-p)
    \left(
        1-\frac{1}{q}
    \right)>0.
\label{eq:delta-eta-comparison}
\end{align}
Consequently,
\[
    0<\eta_q<\delta_{p,q}.
\]
Since $0<s<M$, it follows that
\begin{align*}
    s^{-\eta_q}
    &=
    s^{\delta_{p,q}-\eta_q}\cdot
    s^{-\delta_{p,q}}\\
    &\leq
    M^{\delta_{p,q}-\eta_q}
    s^{-\delta_{p,q}}.
\end{align*}
Thus \eqref{eq:singular-tail-eta-intermediate} yields
\begin{equation}
\label{eq:singular-tail-final-intermediate}
    B_f(s)
    \leq
    C\|f\|_{L^q(\Omega)}
    s^{-\delta_{p,q}}.
\end{equation}

Combining
\eqref{eq:singular-main-intermediate},
\eqref{eq:singular-head-final-intermediate}, and
\eqref{eq:singular-tail-final-intermediate}, we obtain
\begin{equation}\label{eq:singular-pointwise-before-root} 
     \big([H(\nabla u)]^*(s)\big)^{p-1}
    \leq
    C\|f\|_{L^q(\Omega)}
    s^{-\delta_{p,q}},
    \qquad 0<s<M.
\end{equation}
Since $p-1>0$, we may take the power $\frac{1}{p-1}$ on both sides of \eqref{eq:singular-pointwise-before-root} and recall that
\[
    \sigma_{p,q}
    =
    \frac{p-1}{\delta_{p,q}},
\]
we therefore obtain
\[
    s^{1/\sigma_{p,q}}[H(\nabla u)]^*(s)
    \leq
    C\|f\|_{L^q(\Omega)}^{\frac{1}{p-1}},
    \qquad 0<s<M.
\]
Taking the supremum over $s\in(0,M)$ and using the definition of the Lorentz norm give
\begin{align}\label{eq:weak-estamite-iii}
    \|H(\nabla u)\|_{L^{\sigma_{p,q},\infty}(\Omega)}
    =
    \sup_{0<s<M}
    s^{\frac{1}{\sigma_{p,q}}}G^*(s)  \leq
    C\|f\|_{L^q(\Omega)}^{\frac{1}{p-1}}.
\end{align}
This proves \eqref{eq:singular-weak-intermediate}.

Let
\[
    0<\sigma<\sigma_{p,q}.
\]
Using \eqref{eq:singular-pointwise-before-root}, we have
\begin{align*}
    \|H(\nabla u)\|_{L^\sigma(\Omega)}^\sigma
    &=
    \int_0^M
    \bigl([H(\nabla u)]^*(s)\bigr)^\sigma\,ds\\
    &\leq
    C\|f\|_{L^q(\Omega)}^{\frac{\sigma}{p-1}}
    \int_0^M
    s^{-\sigma/\sigma_{p,q}}\,ds\\
    &=
    \frac{
        CM^{1-\frac{\sigma}{\sigma_{p,q}}}
    }{
        1-\frac{\sigma}{\sigma_{p,q}}
    }
    \|f\|_{L^q(\Omega)}^{\frac{\sigma}{p-1}}.
\end{align*}
Finally, we conclude that
\begin{align}\label{eq:strong-estimeta-iii}
    \|H(\nabla u)\|_{L^\sigma(\Omega)}
    \leq
    C_\sigma
    \|f\|_{L^q(\Omega)}^{\frac{1}{p-1}}.
\end{align}
This proves \eqref{eq:singular-strong-intermediate}.

\medskip
\noindent
\emph{Case \textup{(iii)}: $q=n$}

The proof in this section is similar to the proof of part of Theorem \ref{thm:anisotropic-lebesgue-estimates} (iii). By
\eqref{eq:weak-estamite-iii} and definition of $\sigma_{p,q}$, we conclude that
\begin{equation}
\label{eq:singular-qn-weak}
    \|H(\nabla u)\|_
    {L^{\sigma_{p,n},\infty}(\Omega)}
    \leq
    C\|f\|_{L^n(\Omega)}^{\frac{1}{p-1}},
\end{equation}
where
\[
    \sigma_{p,n}
    =
    \frac{n'(p-1)}{2-p}.
\]
And since \eqref{eq:strong-estimeta-iii}, we easily proves the assertion for $q=n$.

\medskip
\noindent
\emph{Case \textup{(iv)}: $q>n$}

We first explain why we can't get an $L^\infty$ estimate directly. The main problem is the first term of $A_f(s)$. We obtain that
\begin{align}
    &s^{- \frac{1}{n'}}
    \|f\|_{L^1(\Omega)}^{2-p}
    A_f(s)^{p-1}
    \notag\\
    &\quad\leq
    s^{-\frac{1}{n'}}
    \left(
        M^{\frac{1}{q'}}\|f\|_{L^q(\Omega)}
    \right)^{2-p}
    \left(
        s^{\frac{1}{q'}}\|f\|_{L^q(\Omega)}
    \right)^{p-1}
    \notag\\
    &\quad=
    M^{\frac{2-p}{q'}}
    \|f\|_{L^q(\Omega)}
    s^{-\frac{1}{n'}+\frac{p-1}{q'}}.
\label{eq:singular-supercritical-head}
\end{align}
Set
\begin{equation}
\label{eq:singular-supercritical-delta}
    \delta_{p,q}
    :=
    \frac{1}{n'}-\frac{p-1}{q'}.
\end{equation}
Then \eqref{eq:singular-supercritical-head} becomes
\begin{equation}
\label{eq:singular-supercritical-head-delta}
    s^{-\frac{1}{n'}}
    \|f\|_{L^1(\Omega)}^{2-p}
    A_f(s)^{p-1}
    \leq
    C\|f\|_{L^q(\Omega)}
    s^{-\delta_{p,q}}.
\end{equation}
If $\delta_{p,q}>0$, then
\[
    s^{-\delta_{p,q}}\longrightarrow\infty
    \qquad\text{as }s\to0^+.
\]
Thus the right-hand side of
\eqref{eq:singular-supercritical-head} is not uniformly bounded on $(0,M)$. Consequently, it does not directly yield $G\in L^\infty(\Omega)$. Unless, we have
\[
\frac{p-1}{q'} \geq \frac{1}{n'}
\]
which means that
\begin{equation*}
    \delta_{p,q}=
    \frac{1}{n'}-\frac{p-1}{q'} \leq 0
\end{equation*}
and $s^{-\delta_{p,q}}$ is upper bounded on $(0,M)$. Hence we can get the same result as degenerate range.
\end{proof}

We first notice that \(\Phi\) is a Young function, since \(\frac{1}{p-1}>1\). Let
\[
    \widetilde \Phi(s)
    =
    \sup_{t\geq 0}\{st-\Phi(t)\},
    \qquad s\geq 0,
\]
be its complementary Young function. For a measurable set \(\Omega\) of
finite measure, we denote the modular associated with \(I\) by
\[
    I_{\Phi}(v)
    =
    \int_{\Omega} \Phi(|v(x)|)\,dx.
\]
The Luxemburg norm in \(L^\Phi(\Omega)\) is defined by
\[
    \|v\|_{L^\Phi}
    =
    \inf\left\{
        \lambda>0:
        I_{\Phi}\left(\frac{v}{\lambda}\right)\leq 1
    \right\},
\]
whereas the Orlicz norm is defined by
\[
    \|v\|_{O^\Phi}
    =
    \sup\left\{
        \int_D |v(x)z(x)|\,dx:
        I_{\widetilde \Phi}(z)\leq 1
    \right\}.
\]
The underlying Orlicz space is the same for the two norms.

\begin{corollary}[Anisotropic Orlicz and Zygmund space]
\label{rem:anisotropic-orlicz-zygmund}
Assume that \(2\leq p<n\), and let \(\Omega\), \(H\), \(f\), and \(u\)
satisfy the hypotheses of Theorem~1.1. 

\smallskip
\noindent
\emph{(i) Orlicz space.}
Let \(\Phi\) and \(\Psi\) be Young functions. Suppose that there exist
constants \(c>0\) and \(t_0>0\) such that
\begin{equation*}
    \Psi(t)\leq \Phi_n(ct)
    \qquad\text{and}\qquad
    \widetilde \Phi(t)\leq (\widetilde \Psi)_n(ct)
    \qquad\text{for every }t\geq t_0.
\end{equation*}
Define the Young function \(E\) by
\begin{equation*}
    E(t)=\Psi\bigl(t^{p-1}\bigr),
    \qquad t\geq0.
\end{equation*}
Then, for every \(f\in L^\Phi(\Omega)\),
\begin{equation*}
    \bigl\|H(\nabla u)\bigr\|_{L^E(\Omega)}
    \leq
    C\,
    \|f\|_{L^\Phi(\Omega)}^{\frac{1}{p-1}}.
\end{equation*}

\smallskip
\noindent
\emph{(ii) Zygmund space.}
In particular, let
\[
    f\in L^q(\log L)^\alpha(\Omega).
\]
The following estimates hold:
\begin{enumerate}
    \item[(a)] If \(q=1\) and \(\alpha>0\), then
    \begin{equation}
    \label{eq:anisotropic-zygmund-q1}
        \bigl\|H(\nabla u)\bigr\|_{
        L^{\frac{n(p-1)}{n-1}}
        (\log L)^{\frac{n\alpha}{n-1}-1}(\Omega)}
        \leq
        C\,
        \|f\|_{L(\log L)^\alpha(\Omega)}^{\frac{1}{p-1}}.
    \end{equation}

    \item[(b)] If \(1<q<n\) and \(\alpha\in\mathbb{R}\), then
    \begin{equation}
    \label{eq:anisotropic-zygmund-subcritical}
        \bigl\|H(\nabla u)\bigr\|_{
        L^{\frac{nq(p-1)}{n-q}}
        (\log L)^{\frac{n\alpha}{n-q}}(\Omega)}
        \leq
        C\,
        \|f\|_{L^q(\log L)^\alpha(\Omega)}^{\frac{1}{p-1}}.
    \end{equation}

    \item[(c)] If \(q=n\) and \(\alpha<n-1\), then
    \begin{equation}
    \label{eq:anisotropic-zygmund-critical}
        \bigl\|H(\nabla u)\bigr\|_{
        \exp L^{\frac{n(p-1)}{n-1-\alpha}}(\Omega)}
        \leq
        C\,
        \|f\|_{L^n(\log L)^\alpha(\Omega)}^{\frac{1}{p-1}}.
    \end{equation}

    \item[(d)] If \(q=n\) and \(\alpha=n-1\), then
    \begin{equation}
    \label{eq:anisotropic-zygmund-double-critical}
        \bigl\|H(\nabla u)\bigr\|_{
        \exp\left(\exp L^{\frac{n(p-1)}{n-1}}\right)(\Omega)}
        \leq
        C\,
        \|f\|_{L^n(\log L)^{n-1}(\Omega)}^{\frac{1}{p-1}}.
    \end{equation}

    \item[(e)] If either \(q=n\) and \(\alpha>n-1\), or
    \(q>n\) and \(\alpha\in\mathbb{R}\), then
    \begin{equation}
    \label{eq:anisotropic-zygmund-supercritical}
        \bigl\|H(\nabla u)\bigr\|_{L^\infty(\Omega)}
        \leq
        C\,
        \|f\|_{L^q(\log L)^\alpha(\Omega)}^{\frac{1}{p-1}}.
    \end{equation}
\end{enumerate}

\end{corollary}

The elliptic results also apply to solutions of the parabolic
equation subject to the additional constraint
\(\partial_t u=\beta(u)\). The following lemma makes this
observation precise.

\begin{corollary}[An application to a class of parabolic equations]
\label{lem:elliptic-parabolic-transfer}
Assume that \(2\leq p<n\), let \(\Omega\subset\mathbb{R}^n\) be open, let \(I=(0,T)\), and
set \(Q=\Omega\times I\). Let
\(\beta:\mathbb{R}\to\mathbb{R}\) be a Borel function.
Suppose that
\[
u,\ f,\ \beta(u)\in L^1_{\mathrm{loc}}(Q)
\]
and that
\begin{equation}
\label{eq:constrained-parabolic-system}
 -\Delta_p^H u +\beta(u) = f,
\qquad
\beta(u) = \partial_t u
\qquad\text{in }\mathcal{D}'(Q).
\end{equation}
Then, for almost every \(t\in I\) there hold the same results as Corollary \ref{thm:anisotropic-lebesgue-estimates}.
\end{corollary}

\smallskip

\vspace{5pt}

\noindent\textbf{Funding Information:}
The first author L. Chen was partly supported by National Natural Science Foundation of China (No. 12622108), a grant from Beijing Institute of Technology (No.
2022CX01002) 


 The third author Y. Yang was supported by Hebei Natural Science Foundation (No. A2025105003) and the Program of China Scholarship Council (No. 202506030179). 

\medskip

\noindent\textbf{Data availability statement:}
As this work is mathematical in nature, there are no associated datasets and all relevant information is available in this manuscript.

\medskip
\noindent
\textbf{Conflicts of Interest Statement:}
The authors have no conflicts of interest to declare.

%
%
%
%


\end{document}